\documentclass[a4paper, reqno]{amsart}

\usepackage[all]{xy}
\usepackage[english]{babel}
\usepackage[utf8]{inputenc}
\usepackage{amssymb,amsmath,amsthm}
\usepackage{amsfonts}
\usepackage{graphicx}
\usepackage{psfrag}
\usepackage[dvipsnames]{xcolor}
\usepackage[hidelinks]{hyperref}
\usepackage{csquotes}
\usepackage[alphabetic]{amsrefs}

\usepackage{cancel}
\usepackage[normalem]{ulem}

\allowdisplaybreaks

\newcommand{\e}{\epsilon}
\newcommand{\s}{\sigma}

\newcommand{\md}{\mathrm d}

\newcommand{\Ker}{\mathrm{Ker}}
\newcommand{\supp}{\mathrm{supp}}
\newcommand{\sign}{\mathrm{sign}}

\renewcommand{\a}{\alpha}
\newcommand{\de}{\delta}
\renewcommand{\b}{\beta}
\renewcommand{\d}{\partial}

\newcommand{\abs}[1]{\left\lvert#1\right\rvert}
\newcommand{\norm}[1]{\left\lVert#1\right\rVert}
\newcommand{\ldr}[1]{\left\langle #1 \right\rangle}

\newcommand{\A}{\mathcal A}
\newcommand{\B}{\mathcal B}
\newcommand{\F}{\mathcal F}
\newcommand{\cH}{\mathcal H}
\newcommand{\M}{\mathcal M}

\renewcommand{\Im}{\mathrm{Im} \,}
\renewcommand{\Re}{\mathrm{Re} \,}

\newcommand{\scl}{{\mathrm{sc}}}
\newcommand{\loc}{{\mathrm{loc}}}
\newcommand{\Hsc}{H_{\scl}}
\newcommand{\Psisc}{\Psi_\scl}

\newcommand{\semi}{\hbar}

\newcommand{\Hscbar}{\bar H_{\scl}}
\newcommand{\Hscdot}{\dot H_{\scl}}

\newcommand{\uf}{\mathfrak u}

\renewcommand{\P}{\mathrm{P}}
\newcommand{\Sph}[0]{\mathbb{S}}	

\theoremstyle{plain}
\newtheorem{thm}{Theorem}[section]
\newtheorem{prop}[thm]{Proposition}
\newtheorem{lemma}[thm]{Lemma}
\newtheorem{cor}[thm]{Corollary}
\theoremstyle{definition}
\newtheorem{definition}[thm]{Definition}

\newtheorem{remark}[thm]{Remark}

\numberwithin{equation}{section}

\newcommand{\R}[0]{\mathbb{R}}							% Real numbers
\newcommand{\C}[0]{\mathbb{C}}							% Complex numbers
\newcommand{\N}[0]{\mathbb{N}}							% Natural numbers
\newcommand{\Z}[0]{\mathbb{Z}}							% Integers
\title[Dual modes in Kerr spacetimes and the Whiting transform]{Dual modes in Kerr spacetimes and the Whiting transform: Mode stability revisited}

\author{Oliver Petersen}
\address{Department of Mathematics, Stockholm University, Albanovägen 28, 10691 Stockholm, Sweden}
\email{oliver.petersen@math.su.se}

\author{Andr\'{a}s Vasy}
\address{Department of Mathematics, Stanford University, CA 94305-2125, USA}
\email{andras@math.stanford.edu}


\begin{document}
\hbadness=100000
\vbadness=100000

\begin{abstract}
  The purpose of the paper is to place Whiting's classical growing mode stability
  argument, extended to real frequencies by Shlapentokh-Rothman for
  the scalar wave equation and by Andersson, Ma, Paganini and Whiting
  in general, in the framework of classical PDE theory. The key steps
  are: a description of the dual or adjoint modes, a singular phase
  space pairing argument which is technically executed via the Fourier
  transform, followed by a standard unique continuation
  result.

  One part of our description of the dual modes connects them
  directly to the standard mode solutions which are smooth over the
  event horizon. In the zero spin case, there is a geometric interpretation of this: quasinormal modes (for non-zero
  real frequencies) which are smooth across the future event horizon can be extended as distributional QNM solutions by 0 through the past event horizon. 
We also describe
  the behavior of mode solutions at the bifurcate sphere.
\end{abstract}

\maketitle

\tableofcontents
\begin{sloppypar}

  \section{Introduction}

  In 1989, Whiting \cite{W1989} showed that there are no quasinormal modes with positive imaginary part for the Teukolsky equation in subextremal Kerr spacetimes.
In 2013, Shlapenkoth--Rothman \cite{S2015} extended Whiting's result to show that
no quasinormal modes with real frequency exist for the wave
equation. Slightly later, Andersson, Ma, Paganini and Whiting
\cite{AMPW2017} extended (with a
different argument) mode stability for the Teukolsky equation.
More recently, Teixeira da Costa extended these results to the extremal case \cite{TdC2020}.
All these results use versions of what is now known as the Whiting
transform, transforming the relevant ODE by very surprising algebraic
identities. 
We remark that other recent papers investigating the use the symmetries of the Teukolsky equation are by Casals and Teixeira da Costa in \cite{CTdC2022} and very recently by Hollands, Ishibashi and Zahn in \cite{HIZ2026}.
These symmetries of the Teukolsky equations were discovered by Aminov, Grassi and Hatsuda in \cite{AGH2022}.

The purpose of this paper is to show that the Whiting transform
actually can be replaced by the standard Fourier transform and
distribution theory, giving a computationally much simpler proof of
mode stability for subextremal Kerr black holes for real
frequencies, which in addition can proceed {\em without} a full
separation of variables. As dual or adjoint modes are a major ingredient of our
argument, we also analyze these, including on the whole spacetime,
even in cases (such as Kerr-de Sitter) in which we cannot prove mode stability.
Also, from this point of view it is clear why the method does not work for the Klein-Gordon equation (indeed, mode stability is known to be false for the Klein-Gordon equation).

Fix two parameters $a \in \R$ and $m > 0$, such that $\abs a \leq m$.
The domain of outer communication in the subextremal (if $\abs a < m$) or extremal (if $\abs a = m$) \emph{Kerr spacetime} is given in Boyer-Lindquist coordinates $(t, r, \phi, \theta)$ by the real analytic spacetime
\[
	M
		:= \R_t \times (r_+, \infty)_r \times S^2_{\phi, \theta},
\]
with real analytic metric
\begin{equation} \label{eq: g}
\begin{split}
	g
		&= (r^2 + a^2 \cos^2(\theta))\left( \frac{\md r^2}{\mu(r)} + \md \theta^2 \right) \\*
		&\quad + \frac{\sin^2(\theta)}{\left(r^2 + a^2 \cos^2(\theta)\right)}\left(a \md t - \left(r^2 + a^2\right)\md \phi\right)^2 \\*
		&\quad - \frac{\mu(r)}{\left(r^2 + a^2 \cos^2(\theta)\right)}\left(\md t - a \sin^2(\theta)\md \phi\right)^2,
\end{split}
\end{equation}
where
\[
	\mu(r)
		= r^2 - 2mr + a^2.
\]
This expression models a black hole centered at $r = 0$ in spherical coordinates.
We call $a$ as the \emph{angular momentum} and $m$ the \emph{mass} of the black hole.
Note that \eqref{eq: g} is not defined at the north and south poles $\theta = 0$ and $\pi$, however, it is straightforward to check that \eqref{eq: g} extends real analytically to the north and the south poles.
Furthermore, the expression \eqref{eq: g} is singular at the roots of $\mu$, given by
\[
	r_\pm
		= m \pm \sqrt{m^2 - a^2}.
\]
The number $r_-$ is the radius of the Cauchy horizon (also called the inner event horizon) and $r_+$ is the radius of the event horizon.
In extremal horizons, where $\abs a = m$, then $r_- = r_+$ and the horizons coincide, and $\mu$ has a double root there.

In this paper, we are considering a certain class of solutions to the Teukolsky equation in the domain of outer communication with prescribed asymptotics as $r \to r_+$ and as $r \to \infty$.

Fix an $s \in \R$.
Following e.g.~\cite{AMPW2017}, the \emph{Teukolsky operator} is given by (with the opposite sign convention relative to \cite{AMPW2017})
\begin{align*}
	\mathrm L_s
		&:= - \d_r \mu(r) \d_r + \frac1{\mu(r)} \left( (r^2 + a^2)\d_t + a \d_\phi - (r - m) s \right)^2 + 4 s (r + i a \cos(\theta) ) \d_t \\*
		&\qquad - \frac1{\sin(\theta)}\d_\theta \left( \sin(\theta) \d_\theta \right) - \frac1{\sin^2(\theta)}\left(a \sin^2(\theta)\d_t + \d_\phi + i s \cos(\theta) \right)^2
\end{align*}
in the Boyer-Lindquist coordinates $(t, r, \phi, \theta)$.

\begin{remark}
For $s = 0$, the Teukolsky operator essentially reduces to the standard scalar wave equation:
\[
	\Box 
		= (r^2 + a^2 \cos^2(\theta))^{-1} \mathrm L_0.
\]
\end{remark}

\begin{definition}[A mode solution]
Fix $\s \in \R$ and $k \in \Z$.
A solution $u$ to $\mathrm L_s v = 0$ of the form
\[
	v(t, r, \phi, \theta)
		= e^{-it \s - ik\phi} u(r, \theta),
\]
where $e^{-ik\phi}u(r, \theta)$ extends to a smooth function on $S^2_{\phi, \theta}$, is called a \emph{mode solution to the Teukolsky equation with parameters $(\s, k)$.}
\end{definition}

Let us now choose a smooth function $H: \R \backslash\{r_-, r_+\} \to \C$, such that
\begin{equation} \label{eq: H}
	H'(r)
		= \frac{i \left( \left( r^2 + a^2 \right) \s + ak \right) + (r - m)s}{\mu(r)}.
\end{equation}
We will work with the following regularity assumption:

\begin{definition}[A quasinormal mode] \label{def: QNM}
Let $u(r, \theta)$ be a mode solution to the Teukolsky equation with parameters $(\s, k) \in \C \times \Z$. 
We say that $u$ is a \emph{quasinormal mode solution with parameters $(\s, k)$} if it has no incoming radiation, which means that
\[
	e^{H(r)} u(r, \theta)
\]
is a smooth function on $[r_+, \infty) \times S^2$ and 
\[
	e^{- H(r)} u(r, \theta)
\]
is conormal at $r = \infty$, which means that there is an $a \in \R$ such that 
\[
	(r \d_r)^m  \left( r^a e^{- H(r)} u(r, \theta) \right)
\]
is uniformly bounded in $(r_+ + 1, \infty) \times S^2$ for every $m \in \N_0$.
\end{definition}

\begin{remark} \label{rmk: optimal decay}
From the assumption of conormality and the fact that $u$ satisfies the Teukolsky equation, one can easily derive much more precise asymptotics at $r = \infty$, cf.\ \cite{AMPW2017}.
This will however not be necessary for the argument.
\end{remark}

The mode stability for the Kerr spacetime is the non-existence of certain modes.
The goal of this paper is to give a simple proof of the following result:

\begin{thm}[Mode stability for real frequencies]\label{thm: QNM vanishes}
Let $s \in \R$. 
If $u$ is a quasinormal mode solution to the Teukolsky equation $\mathrm L_s u = 0$ with parameters $(\s, k) \in  (\R\setminus\{0\}) \times \Z$, then $u = 0$.
\end{thm}

By a standard separation of variables argument, see e.g.~\cite{AMPW2017}, Theorem~\ref{thm: QNM vanishes} follows from the following statement.
For fixed parameters $(\s, k, \lambda, s) \in \R \times \Z \times (0, \infty) \times \R$, define
\begin{equation}\label{eq:separated-op-def}
	\P
		:= - \d_r \mu(r) \d_r + \frac1{\mu(r)} \left( i \left( (r^2 + a^2) \s + a k \right) + (r - m) s \right)^2 - 4 s i r \s + \lambda;
              \end{equation}
              here $\P$ is simply $\mathrm L_s$ acting on the radial
              part of the
              separated solution.
We similarly say that $u: (a, \infty) \to \C$ is conormal at $r = \infty$ if there is an $a \in \R$ such that 
\[
	(r \d_r)^m  \left( r^a e^{- H(r)} u(r) \right)
\]
is uniformly bounded in $(r_+ + 1, \infty) \times S^2$ for every $m \in \N_0$.

\begin{thm}[The ODE version]\label{thm: radial ODE}
If $u: \R \to \C$ is a smooth solution to 
\[
	\P u
		= 0,
\]
such that $e^{H(r)} u(r)$ is a smooth function on $[r_+, \infty)$ and $e^{-H(r)}u(r)$ is conormal at $r = \infty$.
Then $u = 0$.
\end{thm}

{\em Our arguments in fact go through without the separation of variables,}
using only the mode solution property, i.e.\ involving only the
Killing vector fields. Under this assumption one can work directly with the operator \eqref{eq:separated-op-def},
keeping the term
$$
\frac1{\mu(r)} \left( i \left( (r^2 + a^2) \s + a k \right) + (r - m) s \right)^2
  $$
  unchanged, but
replacing $\lambda$ in $\P$ by a ($s,\s$-dependent) elliptic,
self-adjoint for $\s$ real, second
order differential operator $\Lambda$ on
the sphere with positive principal symbol given by that of the
spherical Laplacian, i.e.\ the dual metric function of the round
metric on the sphere (is in particular independent of $a$),
namely\footnote{It is not immediately clear from this expression that
  $\Lambda$ is well-defined since this expression holds in standard
  spherical coordinates in the Kinnersley tetrad trivialization, thus
  away from the north and south poles. However, for integer $s$,
  $$
  \Lambda=\mathrm L_s
		+ \d_r \mu(r) \d_r - \frac1{\mu(r)} \left( (r^2 + a^2)\d_t + a \d_\phi - (r - m) s \right)^2 -4 s r \d_t,
                $$
              with $\d_r,\d_t$ a priori well-defined on sections of the complex
              line bundle given by the pull back of $\B(s)$ on the
              2-sphere. Moreover, in view of the transition maps
              between the trivializations of the bundle in
              stereographic charts valid near the north, resp.\ south,
              poles, $\d_\phi$ becomes $\d_\phi\pm s$ in
              those trivializations, thus extends smoothly to the
              poles. From this perspective, $\d_\phi$ is the
              infinitesimal generator of a circle action on the line
              bundle $\B(s)$ (regarded as an operator on the full
              spacetime). Note that, for $s\neq 0$, $\d_\phi$ is {\em not} given
              by covariant differentiation with respect to any connection since the vector field $\d_\phi$
              vanishes on the sphere at the poles. Hence the right
              hand side is well-defined as a differential operator
              with smooth coefficients, a priori on sections over the
              spacetime, but the expression in the Kinnersley tetrad
              shows that in fact it restricts to an operator acting on
              sections of $\B(s)$. We thank Pascal
              Millet for explaining this bundle picture to us, and we
              refer to \cite{M2024} for further information.}
\begin{equation}\label{eq:Lambda-def}
\Lambda=- \frac1{\sin(\theta)}\d_\theta \left( \sin(\theta) \d_\theta
\right) - \frac1{\sin^2(\theta)}\left(-i\s a \sin^2(\theta) + \d_\phi + i s \cos(\theta) \right)^2 +4 s a\s \cos(\theta).
\end{equation}
Note that at the principal symbol level then the operator $\P$ agrees
with that of $r^2$ times the time-Fourier transformed Schwarzschild d'Alembertian ($a=0$), except for the change in the
first term $-\d_r\mu\d_r$ arising from the different definition of $\mu$.
{\em Throughout the
paper we comment on the mostly very minor changes this causes; the
main difference is that the Fredholm and regularity theory discussed
in Section~\ref{sec:Fredholm} uses the microlocal setup of
\cites{V2013,HiVa2015} more fully.}

As we shall see below, dual or adjoint mode solutions play a major
role in the paper. Since these can be constructed and analyzed on
other spacetimes, such as Kerr-de Sitter, in a completely similar
manner, we comment on this extension in remarks in the main body of the paper,
although the application to mode stability will be missing since our
argument, outlined below, breaks down in that case. We will also
phrase our construction in terms of the spacetime geometry, which is
of interest in both the Kerr and the Kerr-de Sitter settings.

We end this brief introduction by the outline of the proof (in the
separated case, taking $s=0$ for simplicity) which is in
fact rather simple. Later, in Section~\ref{sec:sketch} we give a
higher level ``philosophical'' sketch that puts the argument in a
larger context.

First, in Section~\ref{sec:Fredholm} we recall the Fredholm theory of \cites{V2013,HiVa2015} for
the conjugated operator 	$\A= e^{H(r)}\P e^{-H(r)}$; this
conjugation amounts to imposing smoothness at, and thus, from an
extended perspective that we employ, across, the event horizon for the
solutions by the above regularity requirement (note that the relevant element of $\Ker\A$ are $e^{H(r)}$
times elements of $\Ker\P$), and a particular oscillatory behavior at
null infinity, i.e.\ as $r\to\infty$. The Fredholm theory implies existence of dual solutions
for $\A^*$ on the dual function spaces, which then in particular are
distributions supported in $r\geq r_+$. Technically it is convenient
for us (in part to connect to the original Whiting picture) to work
with the bilinear adjoint, $\A^\dagger$ rather than with $\A^*$; one
can do this by complex conjugating the elements of the kernel. In
fact, $\A^\dagger=e^{-H(r)}\P e^{H(r)}$ in $r>r_+$ and based on this in
Section~\ref{sec:alternative-construction} we give an alternative version of the
construction of these supported in $r\geq r_+$ adjoint solutions
which could be of interest also as the construction and singularity analysis is more ``hands on''
in that case, using basic distribution theory. In either
manner, we obtain a precise description of the singular structure of
elements of the kernel of $\A^\dagger$; in fact the two perspectives
combined give an even more precise result (though this is technically
not needed for us). As an aside, in Section~\ref{sec:extended-Kerr}
we describe the precise behavior of the mode solutions in the
spacetime extended across the event horizons, including the bifurcate
sphere, showing that they are solutions in a full neighborhood of the
closed region of outer communications, supported in the future of the
null-geodesics generating the past event horizon: one sees the direct
mode behavior near the future event horizon, and the adjoint mode near
the past event horizon.

Next, we would like to consider the imaginary part of the pairing of
an adjoint mode $u$ with
$\A^\dagger u$ in a certain localized manner in phase space. (Of course, the pairing a priori vanishes as $\A^\dagger u=0$;
the point is to write this in a different way.) Note that the unlocalized version gives the standard position
space ``boundary pairing'' which proves the non-existence of such
modes except in case of superradiance, see e.g.\ \cite{S2015}, but we
do not need this, although for completeness we recall this in Appendix~\ref{appendix:position-pairing}.

In order to execute this
argument, we Fourier transform $u$ (to $\hat u$), which is allowed as $u$ is a
tempered distribution on $\R$ supported in $r\geq r_+$ and compute its
precise singular structure from knowing that of $u$; this is done in
Section~\ref{sec:Fourier}. In Section~\ref{sec:Fourier} we also conjugate
$\A^\dagger$ by the Fourier transform, denoting the dual, frequency,
variable by $\xi$, and in Section~\ref{sec:frequency-pairing} we observe that this has a
regular singular point at $\xi=-2\s$ where $\hat u$ is actually
smooth. With these observations, for $\s<0$, the pairing of the Fourier transforms
on $(-2\s,\infty)$ (for $\s>0$ we work on $(-\infty,-2\s)$) is easily
computed in Section~\ref{sec:frequency-pairing} as a sum of two absolute value
squares of complex numbers, one from $\xi=-2\s$ and one from the
asymptotic behavior at $\xi=+\infty$. Via our Fourier transform
computation the latter corresponds to the most singular term of $u$ at $r_+$ and
whose vanishing implies the vanishing of the most singular term of $u$
at the event horizon.

Finally, a standard unique continuation result for ODE completes the
proof in Section~\ref{sec:frequency-pairing}.

We finish the paper by explaining the translation of the Whiting
transform to our adjoint solution language in Section~\ref{sec:Whiting}.

\subsubsection*{Acknowledgements}

The authors are very grateful to Peter Hintz for comments on an
earlier version of the manuscript and for suggesting additional
references, and to Pascal Millet for explaining to us the
structure of the line bundle on which the non-separated Teukolsky
operator acts.
O.P. gratefully acknowledges support from the Swedish Research Council under grant number 2021-04269 and from Knut och Alice Wallenbergs Stiftelse under grant number KAW 2021.0239.
A.V.~gratefully acknowledges support from the National Science Foundation under grants number DMS-2247004 and DMS-2553664 and from a Simons Fellowship of the Simons Foundation.

\section{Fredholm theory}\label{sec:Fredholm}

\noindent
For all $r>r_+$, we define the operator\footnote{It is better to
  consider $\A$ as $\mathrm L_s$, which acts on distributional
  sections of the spacetime extended across the horizons, acting on separated
  modes with a factor that extends smoothly
  across the future event horizon, namely $e^{-i\s t_*-ik\phi_*}$,
  factored out; this is how it actually arises in
  \cites{V2013,HiVa2015}. See the discussion in Section~\ref{sec:KdS-Fredholm}.\label{footnote:extended-conjugation}}
\begin{equation}\label{eq:A-def}
	\A
		:= e^{H(r)}\P e^{-H(r)}.
              \end{equation}
              Keeping with the eventual desire to consider the full
              spacetime picture, this means that $\A$ is the operator
              $e^{H(r)}\mathrm L_s e^{-H(r)}$ acting on the separated
              modes, with $e^{-i\s t-ik\phi}$ factored out, but we
              only use this perspective in
              Section~\ref{sec:extended-Kerr}, and it is not needed
              for our discussion of mode stability.

\begin{lemma} \label{le: conjugation}
The operator $\A$ is given by
\begin{align*}
	\A
		= & \ - \d_r \mu(r) \d_r +\d_r \left( i \left( \left( r^2 + a^2 \right) \s + ak \right)  + (r - m)s \right) \\
		& \ +\left( i \left( \left( r^2 + a^2 \right) \s + ak \right) + (r - m)s \right) \d_r - 4 s i r \s + \lambda,
\end{align*}
which extends real analytically across $r = r_+$ to $\R$.
\end{lemma}
\begin{proof}
We compute
\begin{align*}
	e^{H} \P e^{-H}
		= & \ \P + e^{H}[\P, e^{-H}] \\
		= & \ \P - e^{H}[\d_r, e^{-H}]\mu \d_r - e^{H} \d_r \mu [\d_r, e^{-H}] \\
		= & \ \P + H' \mu \d_r + e^{H} \d_r \mu H' e^{-H} \\
		= & \ \P + H' \mu \d_r + \d_r \mu H' - \mu (H')^2,
\end{align*}
which provides the desired formula after inserting $H'(r)$ as in \eqref{eq: H}.
\end{proof}

Following \cites{V2013,HiVa2015}, for $\sigma\in\R\setminus\{0\}$, the direct Fredholm theory imposing the no incoming radiation
conditions above is obtained by considering $\A$ as an operator on
function spaces imposing sufficient smoothness at $r=r_+$ and
oscillatory behavior $e^{2i\sigma r}$ times conormal at
infinity\footnote{The latter is problematic for $\Im\sigma>0$, but
  this is due to the conjugation by $e^{-H}$ with this $H$ is only
  suitable for real $\s$; otherwise one should follow the discussion
  in Section~\ref{sec:KdS-Fredholm}.}: the
former is a direct consequence of the hypotheses of Theorem~\ref{thm:
  radial ODE}, and the latter is as well taking into account that
$H(r)-i\sigma r$ is a logarithmically bounded symbol. In fact, the
Fredholm theory is extremely flexible; the key point is to disallow
the other potential behavior at infinity, which is symbolicity
(oscillation at zero frequency), and sufficiently singular conormal
behavior at $r_+$. One also needs to `cap off' the problem, for
instance by adding a `final Cauchy hypersurface' at $r=r_+-\delta$,
$\delta>0$.

As we shall see imminently, this operator is Fredholm of index 0:
invertibility for large real $\sigma$ follows from semiclassical
theory; the operator is non-trapping in this sense. The adjoint $\A^*$
is acting on (essentially) dual spaces\footnote{Relative to the
  $L^2$-pairing, thus the dual of (microlocal) high regularity is
  (microlocal) low regularity,
  and the dual of an extendible distributional subspace is a supported
  distributional subspace.}; this effectively means that
support in $r\geq r_+$ is imposed, and this time symbolicity at
$r=\infty$ is allowed, while the $e^{2i\sigma r}$ oscillation is
disallowed. Due to index 0, $\A$ is invertible if and only if it has
trivial kernel, which is if and only if $\A^*$ has trivial kernel.

We in fact consider $\A^\dagger$, the bilinear (as opposed to
sesquilinear) pairing adjoint. The kernels of $\A^*$ and  $\A^\dagger$
are conjugate-linear isomorphic via complex conjugation. Note that
from the above computation,
\begin{align*}
	\A^\dagger
		= & \ - \d_r \mu(r) \d_r -\d_r \left( i \left( \left( r^2 + a^2 \right) \s + ak \right)  + (r - m)s \right) \\
		& \ -\left( i \left( \left( r^2 + a^2 \right) \s + ak \right) + (r - m)s \right) \d_r - 4 s i r \s + \lambda.
\end{align*}
This can also be seen from the computation in $r>r_+$:
\[
	\A^\dagger= (e^{H(r)}\P e^{-H(r)})^\dagger=e^{-H(r)}\P^\dagger
        e^{H(r)}=e^{-H(r)}\P
        e^{H(r)},
      \]
      using that $\P^\dagger=\P$ (which is one reason the bilinear
      adjoint can be convenient to use)
and then we obtain the analytic extension. Corresponding to the domain
of the adjoint, again support in $r\geq r_+$ is imposed, and symbolicity at
$r=\infty$ is allowed, while the $e^{-2i\sigma r}$ oscillation is
disallowed. Notice that\footnote{In Section~\ref {sec:extended-Kerr} we give an
  interpretation of this that is analogous to that of
  Footnote~\ref{footnote:extended-conjugation}, namely via extension
  across the {\em past} event horizon.} $\A^\dagger$ is thus $e^{-H(r)}\mathrm L_s
        e^{H(r)}$ acting on the separated modes with $e^{-i\s
          t-ik\phi}$ factored out.

\begin{figure}
\begin{center}
\includegraphics[width=40mm]{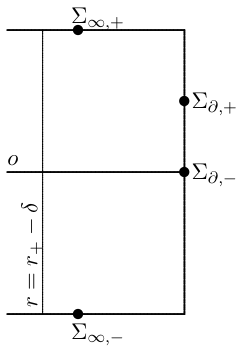}
\end{center}
\caption{The compactified phase space. The horizontal direction is the
real line in $r$ compactified at $+\infty$, the vertical direction is
the real line in $\xi_r$ compactified at $\pm\infty$. The phase space over final Cauchy
hypersurface is the vertical line $r=r_+-\delta$; the zero section $o$
is the horizontal line $\xi_r=0$. The four components
of the characteristic set of $\A$ are $\Sigma_{\infty,\pm}$ at $r=r_+$,
$\xi_r=\pm\infty$ and $\Sigma_{\partial,\pm}$ at $r=+\infty$ with
$\xi_r=2\sigma$, $0$. The figure for $\A^\dagger$ is similar, with the
$\xi_r$ replaced by its negative.}
\label{fig:sct-whiting}
\end{figure}

To see the Fredholm theory, first note that the principal symbol of $\A$ as an operator in $\Psisc^{2,2}$ is
$$
\mu(r)\xi_r^2-2r^2\sigma\xi_r=\xi_r(\mu(r)\xi_r-2r^2\sigma),
$$
and moreover at $r=\infty$ this is equivalent to
$$
r^2\xi_r(\xi_r-2\sigma),
$$
so the characteristic set there consists of the two points $\xi_r=0$
and $\xi_r=2\sigma$, while at fiber infinity this is equivalent to
$$
\mu(r)\xi_r^2,
$$
so as $\xi_r\to\infty$ at fiber infinity, the characteristic set
corresponds exactly to the horizons, where $\mu$ vanishes; for us
$r=r_+$ is the relevant root of $\mu$ as we shall be working in $r\geq
r_+-\delta$, $\delta>0$ sufficiently small so that the only root of
$\mu$ in $r\geq r_+-\delta$ is $r_+$. Thus, in the compactified
approach, the characteristic set consists of four points: $r=r_+$ at
$\xi_r=\pm\infty$, and $\xi_r=0,2\sigma$ at $r=+\infty$; see Figure~\ref{fig:sct-whiting}. At these
points the principal symbol vanishes non-degenerately as a function,
after rescaling, on the boundary of the compactified cotangent bundle
as both $\mu$ and $\xi_r(\xi_r-2\sigma)$ have non-degenerate roots for
$\sigma\neq 0$. The Hamilton vector field is necessarily radial then
(since it annihilates the rescaled principal symbol, and we have a
2-dimensional phase space) and is non-vanishing as a b-vector field,
i.e.\ the characteristic set consists of sources and sinks.

It is instructive to compute the precise nature of the sources and
sinks, as well as the threshold quantities, which arise
from the principal symbol of $\frac{\A-\A^*}{2i}$, adjusted by the
Hamilton vector field applied to the defining function of the relevant
boundary (fiber infinity or spatial infinity) the radial point is at. But for real $\sigma$
\begin{equation}
  \frac{\A-\A^*}{2i}=\frac{1}{i}(\d_r (r-m)s+s(r-m)\d_r)-4sr\sigma,
\end{equation}
so its principal symbol, as an operator in $\Psisc^{1,1}$, at $r=\infty$ is
$$
2sr(\xi_r-2\sigma),
$$
which is $-4sr\sigma$ at the $\xi_r=0$ component and $0$ at the
$\xi_r=2\sigma$ component of the characteristic set,
and at fiber infinity
$$
2s(r-m)\xi_r,
$$
which is $2s(r_+-m)\xi_r$ at the characteristic set.
Moreover $r^{-1}$, resp.\ $|\xi_r|^{-1}$ define spatial, resp.\ fiber
infinity, and
$$
H_ar^{-1}=-2(\xi_r-\sigma)\ \text{at}\ r=\infty,\ \text{resp.}\ H_a|\xi_r|^{-1}=\mu'(r)\sign(\xi_r).
$$
This gives that for $\s>0$
\begin{enumerate}
\item
  the points $\xi_r=0$ at $r=\infty$ and
  $\xi_r>0$, $r=r_+$ at fiber infinity are sources, and
\item
  $\xi_r=2\sigma$
at $r=\infty$ and $\xi_r<0$, $r=r_+$ are sinks for the Hamilton
flow;
\end{enumerate}
for $\s<0$ instead
\begin{enumerate}
\item
  the points $\xi_r=2\sigma$ at $r=\infty$ and
  $\xi_r>0$, $r=r_+$ at fiber infinity are sources, and
\item
  $\xi_r=0$
at $r=\infty$ and $\xi_r<0$, $r=r_+$ are sinks for the Hamilton
flow.
\end{enumerate}

The relevant threshold quantity is then, with the first term
arising from the order of the operator and the second from the
rescaling of the principal symbol of the skew-adjoint part:
\begin{enumerate}
\item
  at $r=\infty$, $\xi_r=0$
$$
\frac{2-1}{2}+\frac{1}{2\sigma}(-4s\sigma)=\frac{1}{2}-2s;
$$
\item
at $r=\infty$, $\xi_r=2\sigma$
$$
\frac{2-1}{2}-\frac{1}{2\sigma}\cdot 0=\frac{1}{2};
$$
\item
at $r=r_+$, both at $\xi_r>0$ and at $\xi_r<0$,
$$
\frac{2-1}{2}+\frac{1}{\mu'(r_+)}(2s(r_+-m))=\frac{1}{2}+2s\frac{r_+-m}{r_+-r_-}=\frac{1}{2}+s.
$$
\end{enumerate}
Given the asymptotics we impose, the orders $\kappa,\ell$ of the Sobolev
space $\Hsc^{\kappa,\ell}$ we are working with need to satisfy
\begin{enumerate}
\item
  $\kappa>\frac{1}{2}+s$ at $r=r_+$
\item
  $\ell>\frac{1}{2}-2s$ at
  $r=\infty$, $\xi_r=0$,
\item
  $\ell<\frac{1}{2}$ at $r=\infty$,
  $\xi_r=2\sigma$.
\end{enumerate}
The operator then acts
$$
\A:\{u\in\Hscbar^{\kappa,\ell}:\ \A u\in\Hscbar^{\kappa-1,\ell-1}\}\to\Hscbar^{\kappa-1,\ell-1},
$$
while on analogous spaces\footnote{More precisely the actual adjoint
  of $\A$ maps $\Hscdot^{-\kappa+1,-\ell+1}$ to
  $\Hscdot^{-\kappa,-\ell}+\A^*\Hscdot^{-\kappa+1,-\ell+1}$, but the
  estimates needed to establish a Fredholm theory are identical.}
$$
\A^*:\{v\in\Hscdot^{-\kappa+1,-\ell+1}:\ \A^*v\in\Hscdot^{-\kappa,-\ell}\}\to\Hscdot^{-\kappa,-\ell},
$$
and for $\A^*$, with $\kappa^*=-\kappa+1$, $\ell^*=-\ell+1$ these threshold inequalities thus amount to
\begin{enumerate}
\item
  $\kappa^*<\frac{1}{2}-s$ at $r=r_+$,
\item
  $\ell^*<\frac{1}{2}+2s$ at
  $r=\infty$, $\xi_r=0$,
\item
  while $\ell^*>\frac{1}{2}$ at $r=\infty$,
  $\xi_r=2\sigma$.
\end{enumerate}
Complex conjugation simply amounts to pull back by
the map $\xi_r\mapsto-\xi_r$, so for $\A^\dagger$ the orders
$\kappa^\dagger,\ell^\dagger$ are $\kappa^*,\ell^*$ pulled back by
this map and thus
\begin{enumerate}
\item
  $\kappa^\dagger<\frac{1}{2}-s$ at $r=r_+$,
\item
  $\ell^\dagger<\frac{1}{2}+2s$ at
  $r=\infty$, $\xi_r=0$,
\item
  while $\ell^\dagger>\frac{1}{2}$ at $r=\infty$,
  $\xi_r=-2\sigma$.
  \end{enumerate}

Since for us the large parameter behavior also matters in $\sigma$,
for establishing invertibility of $\A$ in this case, we note that with
the semiclassical rescaling $h=|\sigma|^{-1}$, and $\xi_{r,\semi}$ the
semiclassical symbol of $hD_r$, the semiclassical
principal symbol of $\A$ is
$$
h^{-2}(\mu(r)\xi_{r,\semi}^2-2(r^2+a^2)\xi_{r,\semi})=h^{-2}\xi_{r,\semi}(\mu(r)\xi_{r,\semi}-2(r^2+a^2)).
$$
Correspondingly, as now at $h=0$ we can have characteristic set in the
interior of the cotangent bundle, one of the components of the characteristic set is
the zero section, $\xi_{r,\semi}=0$. For $r>r_+$ the other component
is $\xi_{r,\semi}=2\mu(r)^{-1}(r^2+a^2)$, which tends to
$\xi_{r,\semi}=+\infty$ as $r\to r_++$, while in $r<r_+$ it is
$\xi_{r,\semi}=2\mu(r)^{-1}(r^2+a^2)$, which tends to $\xi_{r,\semi}=-\infty$
as $r\to r_+-$; these two points lie at fiber infinity and are the
two already discussed points of the characteristic set there. In
agreement with the discussion of the points in the characteristic set
at fiber infinity and $r=\infty$, in the $\xi_r=0$ component we have a
source at $r=\infty$, with the bicharacteristic tending to our final
Cauchy hypersurface at $r=r_+-\delta$, while in the component that
intersects $r>r_+$, we have a source at fiber infinity at $r=r_+$, $\xi_{r,\semi}>0$ and
a sink at $r=\infty$, $\xi_{r,\semi}=2$, while in the component that
intersects $r<r_+$ we have a sink at fiber infinity at $r=r_+$,
$\xi_{r,\semi}<0$, and the bicharacteristic tends here from the final
Cauchy hypersurface at $r=r_+-\delta$. This means that we have
non-trapping semiclassical dynamics and thus large $\sigma$ estimates,
giving the invertibility of $\A$, and thus its index $0$ property, then.

Now, for
$$
\A^\dagger: \{v\in\Hscdot^{\kappa^\dagger,\ell^\dagger}:\ \A^\dagger v\in\Hscdot^{\kappa^\dagger-1,\ell^\dagger-1}\}\to\Hscdot^{\kappa^\dagger-1,\ell^\dagger-1},
$$
subject to the constraints on the orders, the kernel is actually
independent of the orders, i.e.\ lies in the intersection of all these
spaces. In fact, by the results of Haber and Vasy \cite{HaVa2015}, elements of the kernel
are conormal to $r=r_+$ as well as symbolic at $r=\infty$. Here in
fact these regularity statements are much simpler than in
\cite{HaVa2015} since an appropriate elliptic multiple of $\A^\dagger$
spans the microlocal $\Psisc^{0,0}$-submodule of $\Psisc^{1,1}$ consisting of
operators characteristic at the conormal bundle of $r=r_+$, resp.\ at
the zero section at $r=\infty$, so the conormal/symbolic regularity
immediately follows from the regularity (in this case vanishing) of
$\A^\dagger u$.

{\em In case we do not separate variables fully}, rather use the
operator with \eqref{eq:Lambda-def} in place of $\lambda$, the
principal symbol of $\A$ as an operator in $\Psisc^{2,2}$ (on
$\R\times\Sph^2$) becomes
$$
\mu(r)\xi_r^2-2r^2\sigma\xi_r+r^2\nu^2=\xi_r(\mu(r)\xi_r-2r^2\sigma)+r^2\nu^2,
$$
where $\nu$ is the spherical scattering covector variable ($r^{-1}$
times standard spherical covector), and moreover
at
 $r=\infty$ this is equivalent to
$$
r^2\xi_r(\xi_r-2\sigma) +r^2\nu^2,
$$
which is exactly the same as for $r^2$ times the conjugated Schwarzschild d'Alembertian, while at fiber infinity this is equivalent to
$$
\mu(r)\xi_r^2 +r^2\nu^2,
$$
which only differs from the Schwarzschildean version by a different
definition of $\mu$. Now the characteristic set is no longer discrete,
but is still disjoint from $r_+<r<\infty$, and
at infinity the Hamilton flow is exactly the same as in the
Schwarzschild case, while for $r\leq r_+$ has the same qualitative
features as for Schwarzschild. Thus,
for $\s>0$
\begin{enumerate}
\item
  the submanifolds $\xi_r=0,\nu=0$ at $r=\infty$ and
  $\xi_r>0,\nu=0$, $r=r_+$ at fiber infinity are sources, and
\item
  $\xi_r=2\sigma,\nu=0$
at $r=\infty$ and $\xi_r<0,\nu=0$, $r=r_+$ are sinks for the Hamilton
flow;
\end{enumerate}
for $\s<0$ instead
\begin{enumerate}
\item
  the points $\xi_r=2\sigma,\nu=0$ at $r=\infty$ and
  $\xi_r>0,\nu=0$, $r=r_+$ at fiber infinity are sources, and
\item
  $\xi_r=0,\nu=0$
at $r=\infty$ and $\xi_r<0,\nu=0$, $r=r_+$ are sinks for the Hamilton
flow.
\end{enumerate}
Since all the sources and sinks are at $\nu=0$, where the principal
symbol of $\Lambda$ vanishes quadratically, and since $\Lambda$ is
symmetric, $\Lambda$ has no impact at
all on the threshold computations. Thus, all of the above threshold
computations are unchanged, and only $\nu=0$ should be added to the
actual critical set definition on each line for each of the operators
$\A,\A^*,\A^\dagger$. A minor difference is that while the bundle
$\B(s)$ has a Hermitian inner product\footnote{This corresponds to the
  transition maps in the standard trivializations, as in \cite{M2024},
  being multiplication by a factor of absolute value $1$; cf.\ also
  the discussion for $\A^\dagger$ below.}, thus $\A^*$ is well-defined as
acting on this bundle, the complex conjugation needed for defining
$\A^\dagger$ from $\A^*$, defined in the Kinnersley trivialization at
first, does not extend as a map $\B(s)\to\B(s)$, but it {\em does
  extend} to a map $\B(s)\to \B(-s)$, i.e.\ as a map from $\B(s)$ to its
dual bundle $\B(s)$, since under the transition map $z\mapsto
e^{\pm is\phi}z$ complex conjugation becomes $\overline{z}\mapsto
e^{\mp is\phi}\overline{z}$. Since $\A^\dagger$ is the conjugate of
$\A^*$ by complex conjugation, $\A^\dagger$ is well defined as a
map acting on distributional sections of $\B(-s)$.

In addition, the results of Haber and Vasy
\cite{HaVa2015} still apply (this time this is a non-trivial
application of \cite{HaVa2015}), so elements of the kernel of $\A^\dagger$
are conormal to $r=r_+$ as well as symbolic at $r=\infty$. Finally,
for concluding that the index is 0, it suffices to consider $a=0,s=0$
(we allow $s\in\R$!)
since the index is constant under deformations, but then this {\em is}
$r^2$ times
the Fourier transformed (conjugated) Schwarzschild d'Alembertian for
which this has been shown in \cites{V2021a,V2021b}; in particular it
follows from the trivial kernel and cokernel of the $\s=0$ problem by
\cite{V2021b}; see also \cite{HHV21}*{Section 4}.
 (The papers \cites{V2021a,V2021b} use Lagrangian rather
than variable order spaces, but elements of the kernel of the operator
and its adjoint for either setup automatically lie in the other space
by the regularity theory.)

\subsection{Spacetime extension across the horizons in Kerr and Kerr-de Sitter spaces}\label{sec:KdS-Fredholm}
In this section, which is not needed for the mode stability result, we
discuss the Kerr-de Sitter version of the theory for $s=0$; we also use this
opportunity to connect the conjugation in \eqref{eq:A-def} to the coordinate change one
introduces usually both in the Kerr and in the Kerr-de Sitter
setting for extension across the future event horizon.

This coordinate change usually takes the form
\begin{equation}\label{eq:t_*-phi_*}
t_*=t-\Phi(r),\ \phi_*=\phi-\Psi(r),
\end{equation}
and $\Phi,\Psi$ are specified via their derivatives:
\[
\Phi'(r)=b\frac{r^2+a^2}{\mu(r)}f(r),\ \Psi'(r)=b\frac{a}{\mu(r)}f(r),
\]
where $f$ is smooth on a neighborhood of $[r_e,r_c]$, $f(r_e)=-1$,
$f(r_c)=1$, with $r_e$ and $r_c$ the loci of the event, resp.\
cosmological, horizons, so in the Kerr case $r_e$ corresponds to
$r_+$, and $r_c$ to $+\infty$; we refer to \cite{PV2025} for a
description of the Kerr-de Sitter geometry. There is a similar description for Kerr; then $b=1$, and
from a {\em purely analytic} (as opposed to geometric) perspective
one can take $f=-1$ constant,
although $f(r_+)=-1$, $\lim_{r\to\infty} f(r)=1$ corresponds to fully
regular, in the sense of smoothness at the event horizon and
conormality at $r=\infty$; the $f\equiv -1$ choice moves the
desired behavior to $\xi_r=2\s$ as discussed above; we return to this momentarily.

Mode solutions $e^{-i\s t-ik\phi}u_0$, with the ``profile'' $u_0$ annihilated by $\d_t,\d_\phi$, take the form
$$
e^{-i\s t-ik\phi}u_0=e^{-i\s t_*-ik\phi_*} (e^{-i\s\Phi(r)-ik\Psi(r)} u_0),
$$
i.e.\ the modes with respect to the new coordinates are
$e^{-i\s\Phi(r)-ik\Psi(r)}$ times the modes of the previous
form. Correspondingly, acting on the new ``profile'' $e^{-i\s\Phi(r)-ik\Psi(r)} u_0$, $\P$ is replaced by its conjugated version
\begin{equation}\label{eq:KdS-P-conj}
e^{-i\s\Phi(r)-ik\Psi(r)}\P e^{i\s\Phi(r)+ik\Psi(r)}.
\end{equation}
Comparing with the start of the section, this means that (for $s=0$)
\begin{equation}\label{eq:KdS-H}
H=-i\s\Phi(r)-ik\Psi(r),\ H'=-i\s\Phi'-ik\Psi',
\end{equation}
which {\em is}
the choice of $H$ for Kerr in the introduction if $f$ is identically $-1$,
i.e.\ is the ``right'' choice for the event horizon, but the ``wrong''
choice at null infinity (from a compactification
perspective\footnote{Or indeed for considering Fredholm theory for $\Im\s>0$.}, say);
this analogy explains the non-symbolic, rather oscillatory, behavior
of mode solutions for Kerr at null-infinity (as $r\to\infty$). In
terms of the action of $\mathrm L_0$ on separated modes, it is just
$\P$ when $e^{-i\s t-ik\phi}$ is factored out from the mode, i.e.\
\eqref{eq:KdS-P-conj} is just $\mathrm L_0$ acting on a separated
mode, but with $e^{-i\s t-ik\phi}e^{i\s\Phi(r)+ik\Psi(r)}=e^{-i\s
  t_*-ik\phi_*}$ factored out. Since $L_0$ is actually a smooth
differential operator across the horizons and $t_*,\phi_*$ are smooth
across the future event horizon, this means that we $\A$ is simply
$L_0$ acting on modes with $e^{-i\s t_*-ik\phi_*}$ factored out.

For Kerr-de Sitter spacetime, unlike Kerr, there is a significant analytic cost
for making the ``wrong'' choice (beyond considering $\Im\s>0$): the analogue of the scattering
algebra there is the much harder to use (for non-elliptic Fredholm
theory, when the modes are not fully separated) 0-algebra, so it is
best to work across {\em both} horizons.
Thus, we first define the conjugation of $P$ as in
\eqref{eq:KdS-P-conj}, i.e.\ we {\em define} $H$ by \eqref{eq:KdS-H}, with
$f(r_e)=-1$, $f(r_c)=1$. Then
the Fredholm theory we discussed above goes through when we place
final Cauchy hypersurfaces at $r=r_e-\delta_e$ and $r=r_c+\delta_c$,
$\delta_e,\delta_c>0$, $r_e-\delta_e$ greater than $r$ at the Cauchy
horizon. A change is that $r=r_e$ has the source for $\xi_r>0$ but
$r=r_c$ for $\xi_r<0$. For the adjoint operator the final Cauchy
hypersurface become initial Cauchy hypersurfaces, and thus one is
working with spaces of supported distributions, which in particular
implies that the dual mode solutions are supported in $[r_e,r_c]$, and
they are conormal to $r=r_e$ and $r=r_c$.

\section{A sketch of the argument}\label{sec:sketch}
In this section we give a high level and rough sketch of the proof of
our main result by placing it in a larger context; we remark upfront that the sketch is unaffected
except in notation by using the non-separated operator. {\em We
  emphasize up front that this section is not needed for the argument
  of the paper, thus the reader may freely choose to ignore it, but we
  hope that it will be useful for at least some of the readership
  since it shows how the argument connects to microlocal analysis.} The
key point is a positive commutator estimate, which is not so easy to
justify directly, hence needing to go through the explicit Fourier
transform route in the rest of the argument. Thus, here, we work with the operator
$$
\mathcal P
		:= \A^\dagger
                $$
                given by
\begin{align*}
	\mathcal P
		= & \ - \d_r \mu(r) \d_r - \d_r \left( i \left( \left( r^2 + a^2 \right) \s + ak \right)  + (r - m)s \right) \\
		& \ - \left( i \left( \left( r^2 + a^2 \right) \s + ak \right) + (r - m)s \right) \d_r - 4 s i r \s + \lambda.
\end{align*}
In fact, for now we assume $s=0$; we comment on the general case
later.

\begin{figure}
\begin{center}
\includegraphics[width=60mm]{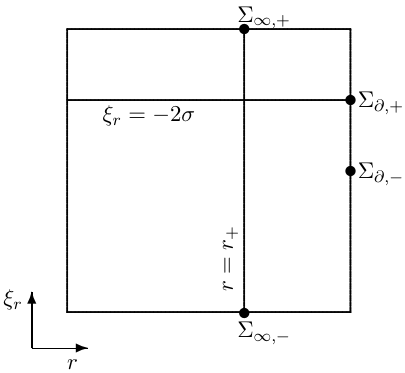}
\end{center}
\caption{The full compactified phase space, compactified as
  $[-\infty,+\infty]$ in both $r$ and $\xi_r$. The final Cauchy
hypersurface does not play a role, so it is not shown, but the phase
space over the event horizon, $r=r_+$, is; the support of $\uf$ is in
$r\geq r_+$, to the right of this line. The horizontal line
$\xi_r=-2\sigma$ is shown; the singular commutant $B_R$ is supported
in $\xi_r\geq-2\sigma$, above this line. The four components
of the characteristic set of $\A^\dagger$ are $\Sigma_{\infty,\pm}$ at $r=r_+$,
$\xi_r=\pm\infty$, and $\Sigma_{\partial,+}$, resp.\ $\Sigma_{\partial,-}$, at $r=+\infty$ with
$\xi_r=-2\sigma$, resp.\ $0$, shown here for $\sigma<0$. Of these
$\Sigma_{\infty,\pm}$ and $\Sigma_{\partial,-}$ contain the wave front
set of $\uf$. Hence, for a singular commutant $B_R$, the contributions
due to the insuffucient regularity of $\uf$ are only from
$\Sigma_{\infty,+}$.}
\label{fig:sct-whiting-full}
\end{figure}

Paralleling the discussion in Section~\ref{sec:Fredholm}, the
principal symbol of $\mathcal P$
in
 $\Psisc^{2,2}$ is
$$
\mu(r)\xi_r^2+2r^2\sigma\xi_r=\xi_r(\mu(r)\xi_r+2r^2\sigma),
$$
and moreover at $r=\infty$ this is equivalent to
$$
r^2\xi_r(\xi_r+2\sigma),
$$
so the characteristic set there consists of the two points $\xi_r=0$
and $\xi_r=-2\sigma$, while at fiber infinity, as $\xi_r\to\infty$, this is equivalent to
$$
\mu(r)\xi_r^2,
$$
so at fiber infinity the characteristic set is
exactly at the horizons. By working with elements $\uf$ of the kernel of $\mathcal
P=\A^\dagger$, the distributions we are interested in are supported in
$r\geq r_+$, and the conormal regularity at $r=r_+$ as well as the
symbolic behavior at $r=\infty$ means that the point in the
characteristic set where our distributions are non-trivial are
$r=r_+$, $\xi_r=\pm\infty$, resp.\ $r=\infty, \xi_r=0$. See Figure~\ref{fig:sct-whiting-full}.

Now, our positive commutator argument\footnote{Of course, negative
  commutator is just as good for our purposes; definiteness is what is important.} takes the following
form. First, given
$\uf\in\Ker\mathcal P$ with the just described behavior, one chooses
a family of operators $B=B_R$, which are order
$-\infty,-\infty$ for finite $R$ at every point in the wave front set
of $\uf$, so all pairings and computations automatically make sense, and
uniformly bounded as $R\to\infty$ with a well behaved limit (so $R$ is
a regularization parameter). Next, using
that $P$ is symmetric (this is the role of $s=0$ for now) one computes
$$
\ldr{i[\mathcal P,B_R]\uf,\uf}=i\ldr{B_R\uf,\mathcal P\uf}-i\ldr{\mathcal P\uf,B_R\uf}=0,
$$
and arranges on the other hand that $i[\mathcal P,B_R]$ is
non-negative and indeed bounded below by a quantity whose vanishing
implies that in fact $\uf$ is trivial (microlocally regular, i.e.\ has
no wave front set) at (at least) one of the points in its a
priori wave
front set. This is the crucial victory after which essentially unique
continuation arguments complete the proof of the vanishing of $\uf$.

A difficulty with this is that the microlocal machinery only allows
one to do the computation modulo compact errors; this reflects that
the principal symbol really ``lives'' at infinity (both base, i.e.\
position, and fiber, i.e.\ momentum, infinity). Another difficulty
is that, due to constraints arising from the Hamilton dynamics that we explain, in fact we need to take $B_R$ to be singular so in fact it is
not a pseudodifferential operator.

In the well-behaved global pairing
arguments in other settings, such as the positivity of propagator
differences paper \cite{V2017}, one uses $B_R$ which actually tends to the identity operator
in a slightly weaker (positive order) space of pseudodifferential operators, uniformly
bounded in order $0,0$ pseudodifferential operators. In this case,
since the identity operator commutes with everything, the only reason
for a non-trivial result is that $\uf$ is in a too large space, namely
it is order $-1/2-\e$ for all $\e>0$ microlocally at some points. If
$\uf$ were actually in $H^{-1/2,-1/2}$, $\ldr{i[\mathcal P,B_R]\uf,\uf}$
would tend to $0$, and thus only the locations where this membership
fails contribute to the result. These are the sources and sinks of the
Hamilton flow at which we need to allow weaker than
the threshold regularity order. Since $B_R$ is a regularizer, its
principal symbol is
decaying at infinity, so if it is non-negative, at sources the
principal symbol of $i$ times the
commutator is positive, at sinks negative since it is the Hamilton vector field
of $P$ applied to the principal symbol of $B_R$. Thus, as long as the wave
front set of $\uf$ is only at sources, or only at sinks, one obtains a
positive (or negative) commutator result, but one cannot mix sources
and sinks, except potentially if the contributions at certain sources
and certain corresponding sinks are coupled (arise from the same
quantity) and either cancel or potentially in combination give the correct sign one
needs for the other terms. (Of course one needs to do a more explicit computation for
the non-trivial sources/sinks to obtain an actually definite result to
conclude the microlocal regularity indicated above.) This is a problem for us since similarly to
Section~\ref{sec:Fredholm}, the Hamilton vector field of the principal
symbol of $\mathcal P$ applied to the
defining function of fiber infinity, $|\xi_r|^{-1}$ is
$\mu'(r)\sign(\xi_r)$, which takes opposite signs at
$\xi_r=\pm\infty$ (with $\xi_r=+\infty$ the source, $\xi_r=-\infty$
the sink), and $\uf$ is singular at both of these at
$r=r_+$. Thus, unless we show and use cancelation from the terms
arising from these two points, we need to work with a commutant that does not tend to
the identity (so that it is supported away from one of these points),
hence the issues raised (non-trivial compact errors, and as we shall
see singular commutants) apply.

Note that at base (position) infinity, where $r^{-1}$ is the defining function,
the Hamilton derivative of this is $-2(\xi_r+\s)$, so at $\xi_r=0$ for
$\s>0$ we have a sink and for $\s<0$ a source. If we only have one
source/sink with non-trivial behavior of $\uf$ in the wave front set of $B$, of course the regularization
gives a definite sign, but we would need to assure that the
localization itself (the fact that the operator does not tend to the
identity) gives a matching contribution. Thus, for instance, for
$\s<0$, from just the perspective of the contributions due to
regularization, we may localize to a region where either $\xi_r$ is greater than a
constant or less than a negative constant; in the former case we have
one or two sources ($r=r_+$, $\xi_r=+\infty$ and possibly $r=\infty$, $\xi_r=0$)
where $\uf$ is a priori non-trivial, in the latter case one sink ($r=r_+$, $\xi_r=-\infty$).

Ignoring the just discussed issues, we work with
a full symbol of $\mathcal P$, namely
$$
p=\mu(r)\xi_r^2-2i\left( i \left( \left( r^2 + a^2 \right) \s + ak \right) + (r - m)s \right)\xi_r - 4 s i r \s + \lambda,
$$
and recall that $s=0$, so some of these terms vanish. In fact, it is
not hard to check that $\mathcal P$ {\em is} the Weyl quantization of
$p+1/2$, but due to the singular symbols below we cannot actually use
the Weyl calculus. Now, with
$b=b_R$ the full symbol of $B$, the principal symbol of $i[\mathcal
P,B_R]$ is $H_pb$. We will however pretend that this is the {\em full}
symbol of the commutator in what follows. Obtaining a positive
commutator thus amounts to finding $b$ that is monotone along the
$H_p$ flow.
Now,\footnote{Using the Weyl calculus we would have $H_{p+1/2}=H_p$ still.}
$$
H_p=\left(2\mu(r)\xi_r+2 \left( \left( r^2 + a^2 \right) \s + ak \right)\right)\d_r-(\mu'(r)\xi_r^2+4r\s\xi_r)\d_{\xi_r}.
$$
Further, $\uf$ has wave front set at both $\xi_r=\pm\infty$ and at
$r=\infty$, but if we localize away from $\xi_r=0$ then the only wave
front set is at $|\xi_r|=\infty$, so we only need to regularize (for
finite $R$) in $\xi_r$ (and not in $r$). This suggests using $b=\psi_R(\xi_r)$; we take
$\psi\geq 0$. We are
then reduced to considering
$$
H_pb=-(\mu'(r)\xi_r^2+4r\s\xi_r)\psi_R'(\xi_r).
$$
Now, as already mentioned, for $\xi_r=+\infty$, $r=r_+$ we have a
source, so if this is included in the support of $b$ the regularizer has a positive Hamilton derivative; we need
to arrange a similar positive derivative elsewhere. But $\psi_R$
should be an increasing function of $\xi_r$ in the localizing region,
where $\xi_r>0$,
i.e.\ we would like $-(\mu'(r)\xi_r^2+4r\s\xi_r)=-(2(r-m)\xi_r^2+4r\s\xi_r)\geq
0$ on the support of $\psi'$. One can cancel the term proportional to
$r$ by choosing $\xi_r=-2\s$ or $\xi_r=0$, which is useful as here we
need to consider all $r\in\R$ and obtain positivity; we do the former to avoid
dealing with the wave front set of $\uf$ at $\xi_r=0$ (already
discussed) and for reasons
related to $s\neq 0$ discussed below. This then
suggests taking $b$ to be $H(\xi_r+2\sigma)$ times the regularizer
($H$ the step function),
which is exactly what we do formally below. Note that for $\s>0$,
taking $H(-\xi_r-2\s)$ works for exactly the same reasons, just the
sources are replaced by sinks and the signs are reversed. We remark
that if we could actually use the Weyl calculus, due to the {\em
  quadratic} polynomial nature of $p$ in $r$, we would in fact be justified in merely
computing $H_pb$ to obtain the full symbol: in the Weyl expansion for
the commutator the next term would be two orders lower in $r$ and be a
polynomial, thus would need to vanish, giving an exact commutator
result. While we do not try to justify this argument, it could be
thought of as the underlying principle.

Now, for $s\neq 0$, a singular conjugation by $\xi_r^s$ restores the
symmetry of $\mathcal P$, but this (the singularity) is one reason that $\xi_r=0$ should
be avoided; once this is done, an argument as above indeed works.

Of course, this argument is rather sketchy since we pretended that we
can perform exact computations in a singular pseudodifferential
operator setting; the purpose of the next sections is to replace this
sketchy argument (which however explains why one expects the argument
to work) by an explicit and precise one.

\section{The Fourier transform of the radial ODE}\label{sec:Fourier}

\noindent
We now work with the operator
\[
	\mathcal P
		:= \A^\dagger;
              \]
              recall that this is $e^{-H(r)}\P e^{H(r)}$ in $r>r_+$,
and
\begin{align*}
	\mathcal P
		= & \ - \d_r \mu(r) \d_r - \d_r \left( i \left( \left( r^2 + a^2 \right) \s + ak \right)  + (r - m)s \right) \\
		& \ - \left( i \left( \left( r^2 + a^2 \right) \s + ak \right) + (r - m)s \right) \d_r - 4 s i r \s + \lambda.
\end{align*}
For our purposes it is important to compute the precise form of
elements $\uf$ of the kernel at $r=r_+$, while the already established
symbolic regularity at $r=\infty$ suffices.

We start by recalling some basic distributions.
For any $a \in \C$ with $\Re(a) > -1$, the function
\[
	x_+^a
		:=
		\begin{cases}
			x^a, & x > 0, \\
			0, & x \leq 0,
		\end{cases}
\]
is locally integrable and can be viewed as the distribution
\[
	x_+^a[\phi]
		:= \int_0^\infty x^a \phi(x) \md x.
\]
This family of distributions is extended to any $a \in \C \backslash (-\N_+)$, by the formula\footnote{For negative integers, $a = - k$, then
\[
	x_+^{-k}[\phi]
		:= - \frac1{(k-1)!} \int_0^\infty \log(x) \d_x^k \phi(x) \md x + \frac1{(k-1)!} \left( \d_x^{k-1}\phi \right) (0) \sum_{j = 1}^{k-1} \frac1j.
              \]}
\[
	x_+^a[\phi]
		:= - \frac1{a + 1}x_+^{a + 1}[\d_x \phi].
\]

              One also defines for $a\in \C \backslash (-\N_+)$
              $$
              \chi_+^a=x_+^a/\Gamma(a + 1);
              $$
              since $\frac{d}{dx}\chi_+^a=\chi_+^{a-1}$ for $a\in \C
              \backslash (-\N_+)$, it is immediate that
              this then extends analytically to the non-positive integers.
We refer to \cite{H2012}*{Sec.~3.2} for a careful discussion of these distributions.

\begin{remark}
For any value of $a \in \C$, the distribution $x_+^a$ coincides with
$x^a$ for all $x > 0$ and vanishes for $x < 0$. On the other hand, for
$a\in (-\N_+)$, $\chi_+^a$ is supported at $\{0\}$, namely $\chi_+^{-k}=\delta_0^{(k-1)}=(d/dx)^kx^0_+$.
\end{remark}

The tempered distributions $(x\pm i0)^a$ also play a role below for a
description of leading order local singularities; each of these has a
one sided wave front set at $\{0\}$. More precisely, we need to work with $x_{\pm
  i0}^a$ which is defined by
$$
x_{\pm i0}^a=(x\pm i0)^a,\ a\notin\N_0,
$$
but as $(x\pm i0)^a$ is a polynomial (thus smooth), for $a\in\N_0$ one
should use
$$
x_{\pm i0}^a=\log(x\pm i0) (x\pm i0)^a.
$$
The key point is that these are classical conormal distributions to
the conormal bundle of $0$, with one sided wave front set, with
homogeneous principal symbol a non-vanishing multiple of
$|\xi|^{-1-a}$ for either $\xi>0$ or $\xi<0$, and vanishing on the
other half line. In fact, they are inverse Fourier transforms, modulo
$C^\infty$, of these functions, smoothed out near $\xi=0$ -- the
behavior for $\xi$ in a compact set is irrelevant for the local
behavior we need here\footnote{But is of course important for the
  global Fourier transform}.

We know by the considerations above, cf.\ the results of Haber and
Vasy \cite{HaVa2015}, but as mentioned this is much simpler here, that at
$r=r_+$ $\uf$ in the kernel of $\A^\dagger$ is conormal relative to the Sobolev space
$H^{\kappa^\dagger}$. Denoting $x=r-r_+$, and
$$
\nu(r)=  \left( \left( r^2 + a^2 \right) \s + ak \right)  -i (r -
m)s,\ \nu_+=\nu(r_+),
$$
$$
\A^\dagger\in D_x x\mu'(r_+) D_x+2 \nu_+ D_x+\M^2=x\mu'(r_+)D_x^2+(2\nu_+-i\mu'(r_+))D_x+\M^2,
$$
where $\M$ denotes the module of first order differential operators
with principal symbol vanishing at $x=0$. Thus, Lemma~6.1 of the
radiation field paper of Baskin, Vasy and Wunsch \cite{BVW2015} is applicable with $v$
there being our $x$, and $\alpha$ there being our $\frac{2\nu_+}{\mu'(r_+)}-i$. The
result then states that for suitable $g_\pm\in \C$,
\begin{equation}\label{eq:horizon-decomp}
\uf=g_+x_{+i0}^{-i \frac{2\nu_+}{\mu'(r_+)}}+g_-x_{-i0}^{-i
  \frac{2\nu_+}{\mu'(r_+)}}+\tilde u,
\end{equation}
where $\tilde u$ is in the conormal space relative to
$H^{\kappa^\dagger+1-\e}$ for all $\e>0$. This result ultimately comes
down to the (here only necessarily leading order, but in fact
complete) classical (i.e.\ one-step polyhomogeneous) conormality of $\uf$, i.e.\ that it is given by the
inverse Fourier transform of a classical symbol; we discuss the action
of the Fourier transform on such distributions imminently. This also
indicates that the $\chi_+^a$ as opposed to the $x_+^a$ distributions are
helpful to work with, since for $a\in(-\N_+)$ (when there is an a priori
difference), it is the former that
necessarily arise from classical symbols. We also note that
$$
\Re\left(-i \frac{2\nu_+}{\mu'(r_+)}\right)=-\frac{2(r-m)}{\mu'(r_+)}s=-s,
$$
and thus
$$
x_{\pm i0}^{-i \frac{2\nu_+}{\mu'(r_+)}}\in\bigcap_{\e>0}H^{\frac{1}{2}-s-\e}.
$$
Since at $r=r_+$, $\kappa^\dagger+1-\e<(1/2-s)+(1-\e)$, choosing
$\kappa^\dagger$ sufficiently close to $1/2-s$ at $r=r_+$, as one may, $\tilde u$
is almost one differential order more regular than the other two terms
in this expansion.

A linear combination as in \eqref{eq:horizon-decomp} can be written in terms of
$\chi_+^{-i \frac{2\nu_+}{\mu'(r_+)}}(x)$ plus just one of the $\pm i0$
distributions, say
\begin{equation}\label{eq:horizon-supp-decomp}
\uf=b\chi_+^{-i \frac{2\nu_+}{\mu'(r_+)}}(x) +g x_{+i0}^{-i \frac{2\nu_+}{\mu'(r_+)}}+\tilde u;
\end{equation}
  for some $b,g\in\C$. Indeed, $(x\pm i0)^a=x_+^a+e^{\pm \pi ia}x_-^a$
  for $\Re a>0$
  shows that
  $$
  e^{-i\pi a}(x+i0)^a-e^{i\pi
    a}(x-i0)^a=(e^{-i\pi a}-e^{i\pi a})x_+^a=e^{-i\pi a} (1-e^{2i\pi a}) \Gamma(a+1)\chi^a_+,
  $$
  and now $(1-e^{2i\pi a}) \Gamma(a+1)$ is holomorphic in $\C$, so the
  equation remains valid for all $a\in\C$.
  Thus
  $$
(x-i0)^a=e^{-2i\pi a}(x+i0)^a-e^{-2i\pi a}\big((1-e^{2i\pi a})\Gamma(a+1)\big) \chi^a_+,
  $$
which deals with the $a\in\C\setminus\N_0$ case as claimed. If
$a\in\N_0$, then
$x_{\pm i0}^a=x^a\log(x\pm i0)$, so
$$
x_{+ i0}^a-x_{-i0}^a=x^a(\log(x+
i0)-\log(x-i0))=x^a(2i\pi)x_-^0=2i\pi x^a-2i\pi x^a_+= 2i\pi x^a-2i\pi\Gamma(a+1)\chi_+^a,
$$
and the first term on the right hand side is smooth,
so a rearrangement deals with this case as well.

  Since by the support conditions $\uf|_{x<0}=0$,
  and as $\supp\chi_+\subset\{x\geq 0\}$, we deduce from \eqref{eq:horizon-supp-decomp} that
  $$
g x_{+i0}^{-i \frac{2\nu_+}{\mu'(r_+)}}|_{x<0}=-\tilde u|_{x<0}\in \bigcap_{\e>0}H^{\kappa^\dagger+1-\e},
  $$
  with the space on the right being extendible distributions at $x=0$.
  This allows us to deduce the following:

  \begin{lemma}\label{lemma:dual-solution-principal-structure}
    For some $b\in\C$
  $$
\uf=b\chi_+^{-i \frac{2\nu_+}{\mu'(r_+)}}(x) +\tilde u,
$$
with $\tilde u$ in the conormal space relative to $\bigcap_{\e>0}H^{\kappa^\dagger+1-\e}$.
Further,
if $b$ vanishes then in fact $\uf\in C^\infty$.
\end{lemma}

\begin{proof}
  To see the vanishing of $g$ in \eqref{eq:horizon-supp-decomp}, it is
  convenient to shift the orders by applying a pseudodifferential
  operator $L\in\Psi^{-i \frac{2\nu_+}{\mu'(r_+)}-\beta}$, indeed a
  Fourier multiplier, by $(\xi-i)^{-i \frac{2\nu_+}{\mu'(r_+)}-\beta}$, that preserves support in
  $\{x\geq 0\}$, and which shifts the exponents and coefficients to
\begin{equation}\label{eq:horizon-supp-decomp-prime}
  \uf'=b'\chi_+^{\beta}(x) +g' x_{+i0}^{\beta}+\tilde u'
\end{equation}
  with $b',g'$ nonzero multiples of $b,g$ (given by the principal
  symbol of $L$ and of the two conormal distributions)
so that $\beta<-1/2$, so that $\chi_+^{\beta}, x_{+i0}^\beta\notin
L^2_{\loc}$, but $\tilde u'\in L^2_{\loc}$, i.e.\
$\kappa^\dagger+1+s+\beta>0$, which for $\kappa^\dagger<\frac{1}{2}-s$ with
the inequality close to equality means $\beta>-3/2$, so the two
inequalities for $\beta$ can be simultaneously satisfied.
Since by the support conditions and support preserving properties of $L$, $\uf'|_{x<0}=0$,
  and as $\supp\chi_+\subset\{x\geq 0\}$, we deduce that
  $$
g' x_{+i0}^{\beta}|_{x<0}=-\tilde u'|_{x<0}\in L^2_{\loc},
$$
where $\loc$ stands for being in $L^2$ in compact subsets of
$(-\infty,0]$ (thus in particular near $0$).
But the left hand side is a non-zero multiple of $g'|x|^{\beta}$, and
$\beta<-1/2$, so this is not in $L^2_{\loc}$ unless $g'=0$. We thus
deduce $g'=0$, hence $g=0$, proving the first claim.
  
The final statement follows from the fact that if $b=0$ then $\uf$ is more regular than 
the threshold regularity, hence the microlocal radial point estimates
give the conclusion.
\end{proof}

The next step is to conjugate the operator $\mathcal P$ with the Fourier transform.
We use the convention
\begin{equation} \label{eq: Fourier transform}
	\F(u)(\xi)
		:= \int_\R e^{-ir\xi} u(r) \md r.
\end{equation}
We define
\[
	\hat {\mathcal P}
		:= \F \mathcal P \F^{-1}.
\]
Our main result concerning this is:

\begin{prop} \label{prop: second conjugation}
For $\xi \neq 0$, we have
\begin{align*}
	\xi^{-s}\hat {\mathcal P} \xi^s
		=& \ - \d_\xi \left( \xi^2 + 2 \s \xi \right) \d_\xi - 2 mi \xi\d_\xi \xi + 2 \xi \left( a^2 \s + a k \right) \\
		& \ + \xi^2 a^2 + s^2 \frac{\xi + 2\s} \xi + \lambda.
\end{align*}
\end{prop}

\begin{remark}
Note that $\xi^{-s}\hat {\mathcal P} \xi^s$ is a formally self-adjoint differential operator, 
which is the key for the boundary pairing. 
\end{remark}

\noindent
As an intermediate step, we compute the Fourier transform of $\mathcal P$.

\begin{lemma}
We have
\begin{align*}
	\hat {\mathcal P}
		= & \ - \d_\xi \left( \xi^2 + 2 \s \xi \right) \d_\xi- 2 mi \xi\d_\xi \xi + s \left( \left( \xi + 2 \s \right) \d_\xi + \d_\xi \left( \xi + 2 \s \right) \right) + 2 ism \xi \\
		& \ + 2 \xi \left( a^2 \s + a k \right) + \xi^2 a^2 + \lambda.
\end{align*}
\end{lemma}
\begin{proof}
With our convention, $\F(\d_r u)(\xi) =  i \xi$ and $\F(r u)(\xi) =  i \d_\xi \F(u)(\xi)$.
Formally replacing all $\d_r$ by $i\xi$ and $r$ by $i \d_\xi$ in the expression for $\mathcal P$ in Lemma~\ref{le: conjugation}, we get
\begin{align*}
	\F \mathcal P \F^{-1}
		= & \ - (i\xi) \left( (i\d_\xi)^2 - 2m(i\d_\xi) + a^2 \right) (i\xi) \\
		& \ - (i\xi) \left( i \left( \left( \left(  i \d_\xi \right)^2 + a^2 \right) \s + ak \right)  + (i \d_\xi - m)s \right) \\
		& \ - \left( i \left( \left( \left( i \d_\xi \right)^2  + a^2 \right) \s + ak \right) + ( i \d_\xi - m)s \right) (i\xi) - 4 s i \s i \d_\xi + \lambda \\
		&= -\xi \left( \d_\xi^2 + 2m i\d_\xi - a^2 \right) \xi + \xi \left( \left( - \d_\xi^2 + a^2 \right) \s + ak  + s \left( \d_\xi + i m \right) \right) \\
		& \qquad +\left(\left( - \d_\xi^2  + a^2 \right) \s +
           ak + ( \d_\xi + im)s \right) \xi + 4 s \s \d_\xi + \lambda.
\end{align*}
We rewrite the highest order part as
\begin{align*}
	- \xi \d_\xi^2 \xi - \s \d_\xi^2 \xi - \s \xi \d_\xi^2
		= & \ - [\xi, \d_\xi] \d_\xi \xi - \d_\xi \xi [\d_\xi, \xi] - \d_\xi \xi^2 \d_\xi \\
		& \ - \s \d_\xi [\d_\xi, \xi] - 2 \s \d_\xi \xi \d_\xi - \s [\xi, \d_\xi] \d_\xi \\
		= & \ - \d_\xi \left( \xi^2 + 2 \s \xi \right) \d_\xi.
\end{align*}
Inserting this proves the assertion.
\end{proof}

\noindent
The final step to get a self-adjoint operator is to conjugate $\hat {\mathcal P}$ by $\xi^s$.

\begin{proof}[Proof of Proposition~\ref{prop: second conjugation}]
The statement follows by noting that
\begin{align*}
	- \xi^{-s} \d_\xi \left( \xi^2 + 2 \s \xi \right) \d_\xi \xi^s
		= & \ - \xi^{-s} [\d_\xi, \xi^s] \left( \xi^2 + 2 \s \xi \right) \d_\xi - \xi^{-s} \d_\xi \left( \xi^2 + 2 \s \xi \right) [\d_\xi, \xi^s] \\
		& \ - \d_\xi \left( \xi^2 + 2 \s \xi \right) \d_\xi \\
		= & \ - s \left( \left( \xi + 2 \s \right) \d_\xi + \d_\xi \left( \xi + 2 \s \right) \right) \\
		&\ - s^2 \frac{\xi + 2 \s}\xi - \d_\xi \left( \xi^2 + 2 \s \xi \right) \d_\xi
\end{align*}
and
\[
	\xi^{-s} 2 mi \xi \d_\xi \xi\xi^s
		= 2 m i s \xi + 2 mi \xi \d_\xi \xi,
\]
and
\[
	\xi^{-s} s \left( \left( \xi + 2 \s \right) \d_\xi + \d_\xi \left( \xi + 2 \s \right) \right) \xi^s
		= 2 s^2 \frac{\xi + 2\s} \xi + s \left( \left( \xi + 2 \s \right) \d_\xi + \d_\xi \left( \xi + 2 \s \right) \right).
\]
\end{proof}

We will also need the detailed behavior of the global Fourier
transform for our solutions.
Given any $k, l \in \R$, the weighted Sobolev spaces on $\R$ are
\[
	H^{k,l}(\R)
		:= \{ u \in \mathcal S'(\R) \mid \ldr{x}^l u \in H^k(\R) \},
\]
with norm 
\[
	\norm{u}_{H^{k, l}(\R)} := \norm{\ldr{x}^lu}_{H^k(\R)}.
\]
We define the Fr\'echet subspaces
\[
	H_*^{k,l}(\R)
		:= \{ u \in \mathcal S'(\R) \mid \left( x \d_x \right)^m u\in H^{k,l}(\R), \ \forall m \in \N_0 \},
\]
with the natural semi-norms.
\begin{lemma} \label{le: Sobolev spaces} 
Let $k, l \in \R$.
The Fourier transform extends to an isometric isomorphism
\[
	\F: H^{k,l}(\R) 
		\to H^{l,k}(\R),
\]
and an isomorphism
\[
	\F: H_*^{k,l}(\R) 
		\to H_*^{l,k}(\R).
\]
\end{lemma}
\begin{proof}
For the first assertion, since
\[
	\left( \F^{-1} \ldr{\xi}^k \F \ldr{x}^l \right) ^*u
		= \ldr{x}^l \F^{-1} \ldr{\xi}^k \F u
\]
for any $u \in \mathcal S$, it follows that
\begin{align*}
	\norm{u}_{H^{k,l}}
		= & \ \norm{\ldr{x}^l u }_{H^k} 
		= \norm{\ldr{\xi}^k \F \ldr{x}^l u }_{L^2} 
		= \norm{\F^{-1} \ldr{\xi}^k \F \ldr{x}^l u }_{L^2}\\
		= & \ \norm{\left( \F^{-1} \ldr{\xi}^k \F \ldr{x}^l \right)^* u }_{L^2}
		= \norm{\ldr{x}^l \F^{-1} \ldr{\xi}^k \F u }_{L^2}
		= \norm{\F u}_{H^{l, k}}.
\end{align*}
Note that
\[
	\F\left( (x \d_x)^m u \right)
		= (- \d_\xi \xi)^m \F u
		= \left( - \xi \d_\xi  - 1 \right)^m \F u
		= \sum_{n = 0}^m (-1)^m {m \choose n} (\xi \d_\xi)^n \F(u)(\xi).
\]
We thus conclude that
\begin{align*}
	\norm{(x \d_x)^m u }_{H^{k,l}}
		= & \ \norm{\F\left( (x \d_x)^m u \right)}_{H^{l,k}} \\
		\geq & \ \norm{\left( \xi \d_\xi \right)^m \F(u)}_{H^{l, k}} - \sum_{n = 0}^{m-1} {m \choose n}\norm{\left( \xi \d_\xi \right)^n \F(u)}_{H^{l,k}},
\end{align*}
for $m \geq 1$.
This proves the second assertion by induction.
\end{proof}

Recall that $\chi_+^a=x_+^a/\Gamma(a + 1)$.

\begin{lemma} \label{le: Fourier transform of homogen}
For any $a \in \C$, 
\[
	\F(e^{-x}\chi_+^a)(\xi)
		= e^{-i(a+1)\pi/2}\left( \xi - i \right)^{-a-1}.
\]
Consequently, $e^{-x}\chi_+^a \in H_*^{\Re(a) + \frac12 - \epsilon, \infty}(\R)$ for all $a \in \C$ and all $\epsilon > 0$.
\end{lemma}

\begin{proof}
Let first $\Re(a) > -1$, so that $e^{-x}x_+^a$ is integrable. 
In that case,
\begin{align*}
	\F(e^{-x}x_+^a)(\xi)
		= & \ \int_0^\infty e^{ix(-\xi + i)} x^a \md x \\
		= & \ \left(i(\xi - i)\right)^{-a} \int_0^\infty e^{ix(-\xi + i)} \left(  ix(\xi - i) \right)^a \md x \\
		= & \ \left( i(\xi - i)\right)^{-a-1} \int_\gamma z^a e^{-z} \md z,
\end{align*}
where the complex contour $\gamma$ is given by
\[
	\gamma(x)
		=  ix(\xi - i) 
		= -i x \xi + x,
\]
defined for $x \in (0, \infty)$.
By the Cauchy integral formula, since $e^{-z}$ is exponentially decaying, we get
\[
	\int_\gamma z^a e^{-z} dz
		= \int_0^\infty s^a e^{-s} ds
		= \Gamma(a+1).
\]
This proves the formula for $\Re(a) > -1$.
The formula now extends by analyticity to all $a \in \C$.
The second assertion now follows from Lemma~\ref{le: Sobolev spaces}.
\end{proof}

We can now compute the Fourier transform of the dual mode solution
$\uf\in\Ker\mathcal P$:

\begin{prop} \label{prop: u Fourier transform}
There is a $M \in \R$, such that
\[	
	e^{ ir_+ \xi}\hat {\mathfrak u}(\xi)
		- be^{-i(\a+1)\pi/2}\left( \xi - i \right)^{-\a-1} 
		\in \bigcap_{\e>0}H_*^{M, \Re(\a) + 3/2-\e}(\R),
\]
where 
\[
	\a
		:= - \frac{i \left( \left( r_+^2 + a^2 \right) \s + ak \right) }{\sqrt{m^2 - a^2}} - s.
\]
\end{prop}

\begin{proof}
By the basic properties of elements of the kernel of $\A^\dagger$ and using Lemma~\ref{lemma:dual-solution-principal-structure}
it follows that 
\[
	\mathfrak u(r)
		= b e^{-(r-r_+)}\chi_+^\a(r-r_+) + \tilde u(r)+k_\infty(r)
              \]
              where on the right hand side the third term $k_\infty$ is a symbol supported in $r>r_+$, lying
              in $\bigcap_{\e>0}H_*^{\infty,1/2+2s-\e}$, encoding the
              asymptotic behavior of $\mathfrak u$ at infinity and
              having no local singularity (in particular at $r_+$),
              the first term is trivial (Schwartz) at $r=\infty$ and
              encodes the leading local singularity, while the second
              term $\tilde u$ is compactly supported in $r\geq r_+$,
              conormal to $r=r_+$,
              and encodes the subleading singularity of $\mathfrak u$
              at $r=r_+$, so after translation of the local
              singularity to $0$, $\tilde u(.+r_+)\in\bigcap_{\e>0} H_*^{1/2-s-\e+1,\infty}$.
With $x := r - r_+$, we can compare the Fourier transforms as
\[
	\F_r(u)(\xi) 
		= \int_\R e^{-i r \xi} u(r) \md r
		= \int_\R e^{-i (x + r_+) \xi} u(x + r_+) \md x
		= e^{-ir_+ \xi} \F_x(u(\cdot + r_+))(\xi).
\]
Using Lemma~\ref{le: Fourier transform of homogen}, we can therefore compute the Fourier transform of the first term in $\mathfrak u$ to be
\begin{align*}
	\F_r \left( be^{-(r-r_+)}\chi _+^\a (r-r_+)\right)
		= & \ be^{-ir_+ \xi} \F_x\left( e^{-x}\chi_+^\a(x) \right)(\xi) \\
		= & \ be^{-ir_+ \xi}  e^{-i(\a+1)\pi/2} \left( \xi - i \right)^{-\a-1}.
\end{align*}
The second and third term lie respectively in $\bigcap_{\e>0} e^{-ir_+
  \xi} H_*^{\infty,3/2-s-\e}$ and
$\bigcap_{\e>0}H_*^{1/2+2s-\e,\infty}$; recall that $\Re \a=-s$.
\end{proof}

{\em In case we do not separate variables}, Lemma~6.1 of the
radiation field paper of Baskin, Vasy and Wunsch \cite{BVW2015} is
still applicable, and the only change to \eqref{eq:horizon-decomp} is
that $g_\pm$ are now complex valued functions on the sphere. In the
proof of
Lemma~\ref{lemma:dual-solution-principal-structure}, in whose
statement $b$ becomes a complex valued function on the sphere, one still uses a
pseudodifferential operator as stated there that preserves supports;
this can be done as in \cite{H2017}*{Appendix~B} (replacing the
cross-section $\R^{n-1}$ by $\Sph^2$, by e.g. using local
coordinates). For the global Fourier transform we can work with the
stated spaces with values in $C^\infty$ functions on the sphere, and
all of the arguments then go through.

\subsection{Kerr-de Sitter changes}
The statement and proof of
Lemma~\ref{lemma:dual-solution-principal-structure} only change in
notation. For the dual solution, in the kernel of
$$
\mathcal
P=\A^\dagger=e^{-H(r)}\P e^{H(r)},
$$
using $H$ as defined in Section~\ref{sec:KdS-Fredholm} (as
$\P^\dagger=\P$ still holds), we can apply the results of Baskin, Vasy
and Wunsch \cite{BVW2015} as above. Thus, with (as $s=0$)
$$
\nu(r)=b((r^2+a^2)\s+ak)f(r),\ \nu_{e/c}=\nu(r_{e/c}),
$$
(really $\nu=-i\mu H'$)
$x=r-r_e$ or\footnote{Technically in the second case at first we take
  $x=r-r_c$, which gives the statement below with that $x$, but then
  redefine $x$ to be its own negative, which effectively switches the
  role of the two terms and multiplies $g_\pm$ by a nonzero constant.} $x=r_c-r$
\eqref{eq:horizon-decomp} becomes near $r_c$ or $r_e$:
\begin{equation}
\uf=g_+x_{+i0}^{-i \frac{2\nu_{e/c}}{\mu'(r_{e/c})}}+g_-x_{-i0}^{-i
  \frac{2\nu_{e/c}}{\mu'(r_{e/c})}}+\tilde u,
\end{equation}
with $g_\pm$, just as $x$, depending on the choice of $e/c$. This
gives that Lemma~\ref{lemma:dual-solution-principal-structure} becomes
\begin{equation}\label{eq:KdS-dual-solution-principal-structure}
\uf=b_{e/c}\chi_+^{-i\frac{2\nu_{e/c}}{\mu'(r_{e/c})}}(x)+\tilde u
\end{equation}
locally near $r_{e/c}$.

\section{An alternative construction of the dual solution}\label{sec:alternative-construction}
 In this section we give an alternative way of constructing the
(bilinear) dual
solution, at first under a non-integrality condition but then removing
the condition; this gives additional insight into the structure of
dual solutions. The dual solutions are
constructed from the solutions of the direct problem via an
appropriate singular multiplication. For self-adjoint problems the (bilinear)
dual solution would be the complex conjugate of the direct
solution. The present problem, at the horizon, relates to the local
behavior of a self-adjoint problem via a conjugation (as well as a
change of the smooth structure and a division), hence such a
conjugation can be expected to show up in the arguments below.

The distributions $x_+^a$ come in naturally when constructing the
\emph{dual solutions} to linear ODE with a regular singular point. We
state the result for general ODEs, and then at the end of the section
we employ it in our particular setting.

\begin{prop}[The intertwining property] \label{prop:intertwining}
Let $0 \in I \subseteq \R$ be an open interval.
Consider the ordinary differential operator
\[
	Q
		:= \d_x \a(x) \d_x + \b(x)\d_x + \d_x \b(x) + \gamma (x),
\]
where $\a, \b, \gamma: I \to \C$ are smooth functions, and $\a(0) = 0, \a'(0) \neq 0$.
We assume that 
\begin{equation}\label{eq:adjoint-non-integrality-ODE}
	\frac{\b(0)}{\a'(0)} \notin \frac12 \N_+.
\end{equation}
Let $Q^\dagger$ denote the transpose operator of $Q$, with respect to the bilinear pairing $\ldr{u, v} := \int_I u(x) v(x) \md x$, i.e.~
\[
	Q^\dagger
		:= \d_x \a(x) \d_x - \left(\b(x)\d_x + \d_x \b(x)\right) + \gamma (x).
\]
Choose a smooth function $f : I \to \C$ satisfying
\[
	f'(x)
		= - 2 \frac{\b(x)}{\a(x)} + 2 \frac{\b(0)}{x \a'(0)},
\]
which extends smoothly to $x = 0$.
Assume that $w: I \to \C$ is a smooth function.
Define
\[
	\mathfrak w
		:= w(x) e^{f(x)} x_+^{-2\frac{\b(0)}{\a'(0)}}.
\]
Then $\mathfrak w$ is a distribution in $I$ with $\supp(\mathfrak w) \subseteq [0, \infty) \cap I$, which is given by $e^{h(x)} w(x)$ for $x > 0$, for a smooth function $h$ satisfying $h'(x) = - 2 \frac{\b(x)}{\a(x)}$, and vanishes for $x < 0$, and 
\begin{equation} \label{eq: Q-Q-dagger-intertwine}
	Q \mathfrak w
		= e^{f(x)} x_+^{-2\frac{\b(0)}{\a'(0)}}Q^\dagger w.
\end{equation}
\end{prop}

An immediate corollary if $Q^\dagger w=0$ is the following:

\begin{cor}[The dual solution] \label{cor: adjoint solution}
  Let $I,Q,Q^\dagger,f$ be as in Proposition~\ref{prop:intertwining}, and suppose that \eqref{eq:adjoint-non-integrality-ODE}
  holds.
Assume that $w: I \to \C$ is a smooth function such that $Q^\dagger w = 0$.
Define
\[
	\mathfrak w
		:= w(x) e^{f(x)} x_+^{-2\frac{\b(0)}{\a'(0)}}.
\]
Then $\mathfrak w$ is a distribution in $I$ with $\supp(\mathfrak w) \subseteq [0, \infty) \cap I$, which is given by $e^{h(x)} w(x)$ for $x > 0$, for a smooth function $h$ satisfying $h'(x) = - 2 \frac{\b(x)}{\a(x)}$, and vanishes for $x < 0$, and 
\begin{equation} \label{eq: Q u distributions special}
	Q \mathfrak w
		= 0.
\end{equation}
\end{cor}

The argument will rely on the following remark.

\begin{remark}[Homogeneity property] \label{rmk: homogeneity}
Define the scaling operator
\[
	M_s f(x)
		:= f(s x),
\]
for any $x \in \R$ and $s > 0$.
We say that a continuous function $f: \R \to \C$ is homogeneous of degree $\lambda \in \C$ if $M_s f(x) = s^\lambda f(x)$ for all $x \in \R$ and $s > 0$.
Since
\[
	\int_\R M_s f(x) \phi(x) \md x
		= \frac1s \int_\R f(x) M_{s^{-1}} \phi(x) \md x,
\]
for any $\phi \in C_c^\infty$, the scaling operator, and the
definition of homogeneity, extends to distributions by the formula
\[
	M_s u[\phi]
		:= \frac1s u[M_{s^{-1}}\phi].
\]
Note that the distributions $x_+^a$ are homogeneous of degree $a \in \C$.
For any $k \in \N_0$, the $k$-th derivative of the Dirac distribution,
$\de^{(k)}$, is homogeneous of degree $-(k+1)$. Note also that
differentiating the homogeneity condition $M_s u=s^\lambda u$ with
respect to $s$ and evaluating at $s=1$ gives
$$
(x\d_x -\lambda)u=0
$$
since
\begin{equation*}\begin{aligned}
(s\frac{d}{ds}M_s u)[\phi]&=-\frac1s u[M_{s^{-1}}\phi]+\frac1s
u[s\frac{d}{ds}M_{s^{-1}}\phi]
=-\frac1s u[M_{s^{-1}}\phi]-\frac1s u[x\mapsto (x\d_x\phi)(s^{-1}x)]\\
&=-\frac1s u[M_{s^{-1}}\phi]-\frac1s u[M_{s^{-1}}(x\mapsto
(x\d_x\phi))]
=-M_su[\phi]-M_s u[(x\mapsto
(x\d_x\phi))]\\
&=-M_su[\phi]+(\d_x xM_s u)[\phi]=(x\d_x M_su)[\phi].
\end{aligned}\end{equation*}
\end{remark}

\begin{proof}[Proof of Proposition~\ref{prop:intertwining}]
Since $x_+^{-2\frac{\b(0)}{\a'(0)}}$ is supported in $x\geq 0$, both
sides of \eqref{eq: Q-Q-dagger-intertwine} are supported in $x\geq 0$,
and thus \eqref{eq: Q-Q-dagger-intertwine} is satisfied for $x < 0$.
Moreover, for $x > 0$, then $x_+^{-2\frac{\b(0)}{\a'(0)}} = x^{-2\frac{\b(0)}{\a'(0)}}$, and hence $\mathfrak w = w(x) e^{h(x)}$, with
\[
	h'(x)
		= \frac{\md}{\md x} \left( f(x) - \frac{2 \b(0)}{\a'(0)}\ln(x) \right)
		= - 2 \frac{\b(x)}{\a(x)},
\]
as claimed.
Hence \eqref{eq: Q-Q-dagger-intertwine} is satisfied for $x > 0$ by the same computation as in the proof of Lemma~\ref{le: conjugation}.
We thus conclude that
$$
Q \mathfrak w-e^{f(x)} x_+^{-2\frac{\b(0)}{\a'(0)}}Q^\dagger w
$$
is a distribution with support at $x = 0$.
Hence it is a linear combination of derivatives of the Dirac distibution, see e.g.~\cite{H2012}*{Thm.~2.3.4}, i.e.~
\begin{equation}\begin{aligned} \label{eq: alpha beta gamma}
	&\left( \d_x \a(x) \d_x + \b(x)\d_x + \d_x \b(x) + \gamma (x) \right) \mathfrak w-e^{f(x)} x_+^{-2\frac{\b(0)}{\a'(0)}}Q^\dagger w\\
		&\qquad= Q \mathfrak w-e^{f(x)} x_+^{-2\frac{\b(0)}{\a'(0)}}Q^\dagger w
		= \sum_{j = 0}^m c_j \de^{(j)},
\end{aligned}\end{equation}
for some $m \in \N_0$ and $c_1, \hdots, c_m \in \C$.
Therefore the two sides of \eqref{eq: alpha beta gamma} are not in $H^{0, \infty}_*(\R)$.
We also know by Lemma~\ref{le: Fourier transform of homogen} that $\psi \mathfrak w \in H_*^{- 2 \Re\left(\frac{\b(0)}{\a'(0)} \right) + \frac12 - \epsilon, \infty}(\R)$ for any $\psi \in C_c^\infty(I)$ and any $\epsilon > 0$.
By Taylor's theorem, we write
\[
	\a(x) = \sum_{j = 1}^N c^\a_j x^j + \a_N(x), \quad \b(x) = \sum_{j = 0}^N c^\b_j x^j + \b_N(x),
\]
\[
	\gamma(x) = \sum_{j = 0}^N c^\gamma_j x^j + \gamma_N(x), \quad w(x) e^{f(x)} = \sum_{j = 0}^N d_j x^j + w_N(x),
      \]
      \[
e^{f(x)}(Q^\dagger w)(x)= \sum_{j = 0}^N e_j x^j + e_N(x),
        \]
for a large $N \in \N$, where $\a_N, \b_N, \gamma_N, w_N,e_N$ vanish to order $N$ at $x = 0$.
We may choose $N$ so large that any term on the left-hand side of \eqref{eq: alpha beta gamma} involving either of $\a_N, \b_N, \gamma_N, w_N,e_N$ is in $H^{0, \infty}_*(\R)$.
Since the right-hand side of \eqref{eq: alpha beta gamma} is not contained in $H^{0, \infty}_*(\R)$, the problem can be reduced to studying
\begin{align*}
	\Bigg( 
		& \d_x \sum_{j = 1}^N c^\a_j x^j \d_x + \sum_{j = 0}^N
           c^\b_j x^j \d_x + \d_x \sum_{j = 0}^N c^\b_j x^j + \sum_{j
           = 0}^N c^\gamma_j x^j \Bigg) \left( \sum_{j = 0}^N d_j x^j
           x_+^{- \frac{2\b(0)}{\a'(0)}} \right)\\
  &\qquad-\sum_{j = 0}^N e_j x^j x_+^{- \frac{2\b(0)}{\a'(0)}} , \\
		= & \ \sum_{j = 0}^m c_j \de^{(j)} + H_*^{0, \infty}(\R).
\end{align*}
By Remark~\ref{rmk: homogeneity}, expanding the left hand side, every term on it is homogeneous of order
\[
	-\frac{2 \b(0)}{\a'(0)} - 1, -\frac{2 \b(0)}{\a'(0)}, \hdots, -\frac{2 \b(0)}{\a'(0)} + 2N.
      \]
      In fact, the situation is even better, for the term with
      homogeneity of $-\frac{2 \b(0)}{\a'(0)} - 1$ in it can only
      arise from the expression on the first line and only by
      taking the $j=1$ summand in the first term, the $j=0$ in the
      second and third terms of the first factor (and no term in the
      last term of the first factor) and the $j=0$ in the second
      factor, but this gives
      $$
      \left( \d_x x\a'(0) \d_x + \b(0)\d_x + \d_x \b(0) \right) x_+^{- \frac{2\b(0)}{\a'(0)}} =0,
      $$
      so the collection of orders to consider is reduced to
      \[
	-\frac{2 \b(0)}{\a'(0)} , -\frac{2 \b(0)}{\a'(0)}+1, \hdots, -\frac{2 \b(0)}{\a'(0)} + 2N.
      \]
None of these orders is a negative integer by the assumption that $\frac{\b(0)}{\a'(0)} \notin \frac12 \N_+$.
This is therefore a contradiction to the homogeneity degrees of the
Dirac distributions (c.f.~Remark~\ref{rmk: homogeneity}) unless $c_1,
\hdots, c_m = 0$. (Explicitly, one can apply a product of first order
operators annihilating the terms on the left hand side, but these give
elliptic multiples of the differentiated delta distributions by the
non-integrality hypothesis, hence the $c_j$ vanish.)
\end{proof}

We can now strengthen Proposition~\ref{prop:intertwining} by replacing
the $x_+^a$ distributions with the $\chi_+^a$; the key difference is
that the $\chi_+^a$ are continuous, and indeed even analytic, in $a$,
with values in dsitrbutions, even for the negative integers $a$ (when
they are differentiated delta distributions). We state this as a
proposition:

\begin{prop}[The strong intertwining property] \label{prop:strong-intertwining}
Let $0 \in I \subseteq \R$ be an open interval.
Consider the ordinary differential operator
\[
	Q
		:= \d_x \a(x) \d_x + \b(x)\d_x + \d_x \b(x) + \gamma (x),
\]
where $\a, \b, \gamma: I \to \C$ are smooth functions, and $\a(0) = 0, \a'(0) \neq 0$.
Let $Q^\dagger$ denote the transpose operator of $Q$, with respect to the bilinear pairing $\ldr{u, v} := \int_I u(x) v(x) \md x$, i.e.~
\[
	Q^\dagger
		:= \d_x \a(x) \d_x - \left(\b(x)\d_x + \d_x \b(x)\right) + \gamma (x).
\]
Choose a smooth function $f : I \to \C$ satisfying
\[
	f'(x)
		= - 2 \frac{\b(x)}{\a(x)} + 2 \frac{\b(0)}{x \a'(0)},
\]
which extend smoothly to $x = 0$.
Assume that $w: I \to \C$ is a smooth function.
Define
\[
	\mathfrak w
		:= w(x) e^{f(x)} \chi_+^{-2\frac{\b(0)}{\a'(0)}}.
\]
Then $\mathfrak w$ is a distribution in $I$ with $\supp(\mathfrak w) \subseteq [0, \infty) \cap I$, which is given by $e^{h(x)} w(x)$ for $x > 0$, for a smooth function $h$ satisfying $h'(x) = - 2 \frac{\b(x)}{\a(x)}$, and vanishes for $x < 0$, and 
\begin{equation} \label{eq:gen-Q-Q-dagger-intertwine}
	Q \mathfrak w
		= e^{f(x)} \chi_+^{-2\frac{\b(0)}{\a'(0)}}Q^\dagger w.
\end{equation}
\end{prop}

\begin{proof}
Under the condition \eqref{eq:adjoint-non-integrality-ODE}, i.e.\
$\frac{\b(0)}{\a'(0)} \notin \frac12 \N_+$, this is the content of
Proposition~\ref{prop:intertwining} since
$\chi_+^{-2\frac{\b(0)}{\a'(0)}}$ from $x_+^{-2\frac{\b(0)}{\a'(0)}}$
by a non-singular, non-vanishing $\Gamma(-2\frac{\b(0)}{\a'(0)}+1)$ factor. To extend the result to the
remaining case, consider the family of operators $Q_z$ depending on a
parameter $z$, given by $\a,\gamma$ independent of $z$, and
$\b_z=\b+z$. Then for $|z|$ small and non-zero,
\eqref{eq:adjoint-non-integrality-ODE} is satisfied for $Q=Q_z$, thus
\eqref{eq:gen-Q-Q-dagger-intertwine} holds then. But both sides of
\eqref{eq:gen-Q-Q-dagger-intertwine} are continuous (in $z$) in the
distributional topology, i.e.\ for $\phi\in C^\infty_c(I)$, applying both sides
to $\phi$, they are both continuous complex-valued functions, so the
equality for $z=0$ also follows.
  \end{proof}

  We again have an immediate corollary:
  
\begin{cor}[The dual solution] \label{cor:gen-adjoint-solution}
  Let $I,Q,Q^\dagger,f$ be as in Proposition~\ref{prop:strong-intertwining}.
Assume that $w: I \to \C$ is a smooth function such that $Q^\dagger w = 0$.
Define
\[
	\mathfrak w
		:= w(x) e^{f(x)} \chi_+^{-2\frac{\b(0)}{\a'(0)}}.
\]
Then $\mathfrak w$ is a distribution in $I$ with $\supp(\mathfrak w) \subseteq [0, \infty) \cap I$, which is given by $e^{h(x)} w(x)$ for $x > 0$, for a smooth function $h$ satisfying $h'(x) = - 2 \frac{\b(x)}{\a(x)}$, and vanishes for $x < 0$, and 
\begin{equation} \label{eq: Q u distributions}
	Q \mathfrak w
		= 0.
\end{equation}
\end{cor}

We finally apply Corollary~\ref{cor:gen-adjoint-solution} with
\begin{align*}
	Q
		= & \ \mathcal P=\A^\dagger, \\
	x
		= & \ r - r_+, \\
	\a(r - r_+)
		= & \ - \mu(r), \\
	\b(r - r_+)
		= & \ - \left( i \left( \left( r^2 + a^2 \right) \s + ak \right)  + (r - m)s \right).
\end{align*}
For this, we need to compute 
\begin{align*}
	\frac{\b(0)}{\a'(0)}
		= & \ \frac{- \left( i \left( \left( r_+^2 + a^2 \right) \s + ak \right)  + (r_+ - m)s \right)}{- \mu'(r_+)} \\
		= & \ \frac{ i \left( \left( r_+^2 + a^2 \right) \s + ak \right) }{2 \sqrt{m^2 - a^2}} + \frac s2;
\end{align*}
we also remark that by \eqref{eq: H},
$$
H'(r)=\frac{\b(r-r_+)}{\a(r-r_+)}.
$$
Define
\[
	v(r)
		: = e^{H(r)}u(r),
\]
which by assumption extends smoothly to $[r_+, \infty)$ and satisfies $\mathcal P^\dagger v = 0$.
Applying Corollary~\ref{cor: adjoint solution} with 
\begin{align*}
	w(r-r_+)
		:= & \ v(r), \\
	f(r - r_+)
		:= & \ -2H(r) + \left( \frac{ i \left( \left( r_+^2 + a^2 \right) \s + ak \right) }{\sqrt{m^2 - a^2}} + s \right) \ln(r-r_+),
\end{align*}
implies that
\[
	\mathfrak w(r - r_+)
		= v(r) e^{f(r-r_+)}(r-r_+)_+^{- \frac{ i \left( \left( r_+^2 + a^2 \right) \s + ak \right) }{\sqrt{m^2 - a^2}} - s}
		=: \mathfrak u(r)
\]
satisfies
\[
	\mathcal P \mathfrak u
		= 0,
\]
in the distributional sense on all of $\R$.
Note that we may choose the primitive function $H$ in \eqref{eq: H} so
that $f(0) = 0$. As
$$
f(r-r_+) - \left(\frac{ i \left( \left( r_+^2 + a^2 \right) \s + ak
    \right) }{\sqrt{m^2 - a^2}} - s\right) \ln(r-r_+)=-2H(r),\ r>r_+,
$$
we have $\mathfrak u=e^{-2H(r)}v=e^{-H(r)}u$; this also agrees with the conclusion
of Corollary~\ref{cor: adjoint solution} since $h'=-2H'$.

Since $v$ is smooth near $r=r_+$ and $f$ is smooth near $0$,
$\mathfrak u$ is conormal to $r=r_+$ and is supported in $r\geq
r_+$. Moreover, as $e^{-H(r)}u$ is conormal to $r=\infty$ by
assumption, $\mathfrak u$ is conormal to $r=\infty$.
Thus, $\mathfrak u$ is in the spaces as
required for the domain of $\A^\dagger$ (supported at the artificial
Cauchy hypersurface and sufficiently regular at $r=\infty$, $\xi_r=-2\sigma$), and indeed agrees with the structure of the dual
solution $u$
demonstrated in Lemma~\ref{lemma:dual-solution-principal-structure}.

\begin{remark}\label{remark:stronger-dual-structure}
A striking consequence of this result is that for
$-2\frac{\b(0)}{\a'(0)}$ a negative integer the dual states are
necessarily differentiated delta distributions supported at
$x=0$. This is much stronger than the conclusion of Lemma~\ref{lemma:dual-solution-principal-structure}, even
in the strengthened version where $b$ is allowed to be smooth and
$\tilde u$ is smooth, supported in $x\geq 0$, for it states that the
$\tilde u$ term vanishes identically, which is not clear from purely
microlocal arguments.

Another fact that is immediate from this approach is that
Lemma~\ref{lemma:dual-solution-principal-structure} can be
strengthened to:
    for some $b\in C^\infty$
  $$
\uf=b\chi_+^{-i \frac{2\nu_+}{\mu'(r_+)}}(x).
$$
In principle this could be deduced from
Lemma~\ref{lemma:dual-solution-principal-structure} by an iterative
argument, essentially determining the symbolic expansion of the
conormal distribution $u$ step by step, but the present approach makes
the conclusion immediate.
\end{remark}

{\em In case we do not separate variables, the arguments presented
here all go through by adding smooth dependence on the spherical
variables.}

\subsection{The Kerr-de Sitter case}\label{sec:KdS-construction}
Only the final part of the argument of this section is affected by
going to the Kerr-de Sitter case, and there the changes are
essentially notational. The final conclusion is that
\[
	\mathfrak u(r)
		= v(r) e^{f(r)} (r-r_e)_+^{- \frac{ 2i b\left( \left( r_e^2 + a^2 \right) \s + ak \right) }{\mu'(r_e)} }(r_c-r)_+^{\frac{ 2ib \left( \left( r_c^2 + a^2 \right) \s + ak \right) }{\mu'(r_c)} }
              \]
              with
              $$
	f(r)
		:= \ -2H(r) + \frac{ 2i b\left( \left( r_e^2 + a^2 \right) \s + ak \right) }{\mu'(r_e)} \ln(r-r_e)-\frac{2 i b\left( \left( r_c^2 + a^2 \right) \s + ak \right) }{\mu'(r_c)} \ln(r_c-r),
              $$
              provides the desired adjoint solution.

\section{Mode solutions on the extended Kerr spacetime}\label {sec:extended-Kerr}
In this section we discuss the precise behavior of mode solutions of
the wave equation on Kerr spacetime extended across the future and
past horizons as well as the bifurcate sphere for $s=0$. Recall from the end of
Section~\ref{sec:Fredholm} that mode
solutions are of the
form
$$
u=e^{-i\s t-ik\phi}u_0
$$
and can be written as
$$
u=e^{-i\s t_*-ik\phi_*} u_*,\qquad u_*=e^{-i\s\Phi(r)-ik\Psi(r)}u_0=e^{H(r)}u_0,
$$
with $u_*$ smoothly extending across $r=r_+$ as a solution of the
conjugated wave equation\footnote{I.e.\ the resulting extension of $u$ solves
the wave equation across $r=r_+$: as discussed at the end of
Section~\ref{sec:Fredholm}, $\A$ is $e^{H(r)}\mathrm L_0
        e^{-H(r)}$  in $r>r_+$ acting on the separated modes, with
        $e^{-i\s t-ik\phi}$ factored out, or better yet $\mathrm L_0$
        acting on separated modes with $e^{-i\s t_*-ik\phi_*}$
        factored out, with the latter description valid across the
        future event horizon.}; here 
\begin{equation}
t_*=t-\Phi(r),\ \phi_*=\phi-\Psi(r),
\end{equation}
and $\Phi,\Psi$ are specified via their derivatives:
\begin{equation}
\Phi'(r)=b\frac{r^2+a^2}{\mu(r)}f(r),\ \Psi'(r)=b\frac{a}{\mu(r)}f(r),
\end{equation}
and in our Kerr case $b=1$, and $f$ is constant $-1$. In the Kerr-de
Sitter case analogous arguments work, but then
$f$ is smooth on a neighborhood of $[r_e,r_c]$, $f(r_e)=-1$,
$f(r_c)=1$. Also recall from Section~\ref{sec:Fredholm}  that
\begin{equation}\label{eq:KdS-H-recall}
H=-i\s\Phi(r)-ik\Psi(r),\ H'=-i\s\Phi'-ik\Psi'.
\end{equation}
Note that, restricting to the Kerr case in notation,
$$
\Phi'=-\frac{r^2+a^2}{\mu}=-\frac{r_+^2+a^2}{\mu'(r_+)(r-r_+)}+\tilde\Phi'(r),
$$
and
$$
\Psi'=-\frac{a}{\mu}=-\frac{a}{\mu'(r_+)(r-r_+)}+\tilde\Psi',
$$
with $\tilde\Phi',\tilde\Psi'$ smooth across $r=r_+$.

As shown in Section~\ref{sec:alternative-construction}, when
$\alpha$ is not a negative integer\footnote{If $\alpha$ is a negative
  integer, the modes are supported on $r=r_+$, so the
  situation is rather different.} (which is satisfied in our case:
$\alpha$ is pure imaginary),
so $\chi_+^\alpha$ and $x_+^\alpha$ differ by a finite non-zero factor,
the adjoint
solutions\footnote{These are denoted by $\uf$ earlier; here we use
  a new notation $u_{**}$ to emphasize the relationship of the modes
  and dual modes to the spacetime.} in fact arise by
considering
$$
u_{**}=e^{-2H(r)}u_* =(r-r_+)^\alpha e^{2i\s\tilde\Phi+2ik\tilde\Psi}u_*,
$$
extending $u_{**}$ across $r=r_+$ as
$$
u_{**}=(r-r_+)^\alpha_+ e^{2i\s\tilde\Phi+2ik\tilde\Psi}u_*,
$$
which as shown in Section~\ref{sec:alternative-construction} solves
the adjoint conjugated equation: $\A^\dagger$ is $e^{-H(r)}\mathrm L_0
        e^{H(r)}$ acting on the separated modes, with $e^{-i\s
          t-ik\phi}$ factored out.

Turning to the spacetime (see Figure~\ref{fig:mode-extension-2}), the past version of our horizon adapted
coordinates are
\begin{equation}
t_{**}=t+\Phi(r),\ \phi_{**}=\phi+\Psi(r),
\end{equation}
with $\Phi,\Psi$
as above. Thus, $\A^\dagger$ is $\mathrm L_0$ acting on
        separated modes with $e^{-i\s
          t_{**}-ik\phi_{**}}$ factored out.
Hence, with
$$
 u_{**}=e^{2i\s\Phi(r)+2ik\Psi(r)}u_*=e^{i\s\Phi(r)+ik\Psi(r)}u_0
 $$
 in $r>r_+$,
we have
$$
u=e^{-i\s t_*-ik\phi_*} u_*,=e^{-i\s t_{**}-ik\phi_{**}}u_{**}.
$$
This expression, together with the above extension of $u_{**}$ as a
supported distribution, shows that $u$ extends across the past event horizon
as a supported
distribution. Moreover, as $\mathrm L_0$,
        extends smoothly across the past event
        horizon, this in particular states
        that the adjoint modes extend to solve the wave equation since
        the action of $\A^\dagger$ on modes even across the past event
        horizon is that of $\mathrm L_0$ with $e^{-i\s
          t_{**}-ik\phi_{**}}$ factored out.

\begin{figure}[ht]
\begin{center}
\includegraphics[width=65mm]{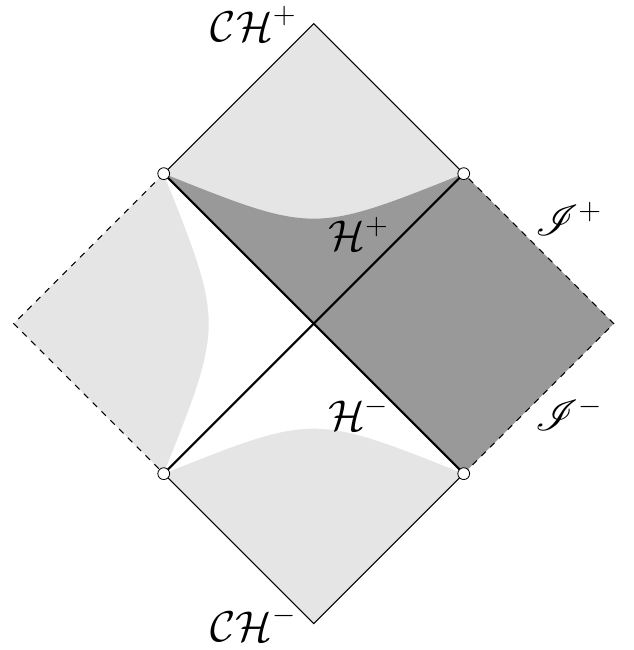}
\end{center}
\caption{The dark region, crossing the future event horizon $\cH^+$, is where the QNM is smooth. It extends to the
  white region, across the past event horizon $\cH^-$ as a
  distributional solution of the wave equation by 0. Note that in
  particular it is a distributional solution in a full neighborhood of
the bifurcate sphere.}
\label{fig:mode-extension-2}
\end{figure}

We want to now study the behavior at the
bifurcate sphere, at first taking $a=0$. Coordinates nearby are given
by the spherical coordinates,
and powers of $e^{-t_{**}}$,
$e^{t_*}$, namely (even if $a$ appears in the equations below for
comparison with the above computations, here we are taking $a=0$)
$$
x_{**}=e^{-\kappa t_{**}},\ x_*=e^{\kappa t_*},\qquad \kappa=\frac{\mu'(r_+)}{2(r_+^2+a^2)}.
$$
with these last two coordinates extended so they can become negative,
and they define the future ($x_{**}=0$), resp.\ past ($x_*=0$), event horizon.
Then in $r>r_+$
\begin{equation}\label {eq:r-t_*-t_**}
e^{-\kappa t_{**}} e^{\kappa
  t_*}=e^{\kappa(t_*-t_{**})}=e^{-2\kappa\tilde\Phi(r)}(r-r_+)^{\kappa \frac{2(r_+^2+a^2)}{\mu'(r_+)}}=e^{-2\kappa\tilde\Phi(r)}(r-r_+),
\end{equation}
with the first factor on the right hand side smooth across $r=r_+$,
which explains the choice of $\kappa$ as the power of $e^{-t_{**}}$,
$e^{t_*}$: $r-r_+$ is a smooth non-degenerate (positive) multiple of the defining
functions $e^{\kappa t_*}$ of the past and $e^{\kappa t_{**}}$ of the
future event horizons. In particular, when these two defining
functions are extended as the coordinates $x_*,x_{**}$ across the horizons, $r-r_+$ remains
a smooth function of these (as $r-r_+\mapsto
e^{-2\kappa\tilde\Phi(r)}(r-r_+)$ is a diffeomorphism near $0$).
Turning to $u$ again, we rewrite $u_{**}$ in an equivalent form along
the past event horizon where $t_{**}$ is finite using
\eqref{eq:r-t_*-t_**}, replacing $e^{\kappa t_*}$ by $x_*$ ($\cdot_+^\alpha$
stands for the $x^\alpha_+$ distribution):
\begin{equation*}\begin{aligned}
u_{**}&=(x_*) _+^\alpha e^{-\alpha\kappa t_{**}}  e^{2\alpha\kappa\tilde\Phi(r)}
e^{2i\s\tilde\Phi+2ik\tilde\Psi}u_*\\
&=(x_*) _+^\alpha e^{i\s t_{**}}u_*,
\end{aligned}\end{equation*}
and (recall $a=0$, so $\phi=\phi_*=\phi_{**}$)
$$
u=e^{-i\s t_{**}-ik\phi}(x_*) _+^\alpha e^{i\s t_{**}}u_*=e^{-ik\phi}(x_*) _+^\alpha u_*.
$$
In view of the above discussed coordinates at the bifurcate sphere, it
is immediate that $u$ extends to a neighborhood of the bifurcate
sphere as a distribution, supported in $x_*\geq 0$ (i.e.\ the continuation of the past event horizon as a
union of nullgeodesics), conormal at $x_*=0$ with the singularity being that of $(x_*)_+^\alpha$.

In fact, it is not hard to deal with the $a\neq 0$ case either. For
this, one needs to change to new coordinates (see \cite{PV23} for
similar considerations)
$$
t_*'=t_*,\ \phi_*'=\phi_*-\frac{a}{r_+^2+a^2}t_*,
$$
and similarly
$$
t_{**}'=t_{**},\ \phi_{**}'=\phi_{**}-\frac{a}{r_+^2+a^2}t_{**}.
$$
Then
$$
e^{-i\s t_*-ik\phi_*}=e^{-i(\s+k\frac{a}{r_+^2+a^2})t_*'}e^{-ik\phi_*'},
  $$
  and
$$
e^{-i\s t_{**}-ik\phi_{**}}=e^{-i(\s+k\frac{a}{r_+^2+a^2})t_{**}'}e^{-ik\phi_{**}'},
$$
so the $\d_{t_*'}$ and $\d_{t'_{**}}$ mode becomes
$$
\s'=\s+k\frac{a}{r_+^2+a^2}
$$
as these vector fields are $\d_{t_*}+\frac{a}{r_+^2+a^2}\d_{\phi_*}=\d_{t_{**}}+\frac{a}{r_+^2+a^2}\d_{\phi_{**}}$.
Then
\begin{equation*}\begin{aligned}
    \phi_{**}'&=\phi_{**}-\frac{a}{r_+^2+a^2}t_{**}=\phi_*-\frac{2a}{\mu'(r_+)}\log(r-r_+)+2\tilde\Psi-\frac{a}{r_+^2+a^2}t_{**}\\
    &=\phi_*'+\frac{a}{r_+^2+a^2}t_*-\frac{2a}{\mu'(r_+)}\log(r-r_+)+2\tilde\Psi-\frac{a}{r_+^2+a^2}t_{**}\\
    &=\phi_*'+\frac{a}{r_+^2+a^2}(-2\Phi)-\frac{2a}{\mu'(r_+)}\log(r-r_+)+2\tilde\Psi\\
&=\phi_*'+2\frac{a}{\mu'(r_+)}\log(r-r_+)-2\frac{a}{r_+^2+a^2}\tilde\Phi
-\frac{2a}{\mu'(r_+)}\log(r-r_+)+2\tilde\Psi\\
&=\phi_*' -2\frac{a}{r_+^2+a^2}\tilde\Phi +2\tilde\Psi,
    \end{aligned}\end{equation*}
which shows that $\phi_*'$ and $\phi_*''$ are smoothly related so
either can be used in place of the other. Thus
$$
u=e^{-i\s' t_*'-ik\phi_*'} u_*=e^{-i\s' t_{**}'-ik\phi_{**}'}u_{**},
$$
with $u_*,u_{**}$ as before. Also, we have
$$
\alpha=-2i\s'\frac{r_+^2+a^2}{\mu'(r_+)}.
$$

Then the above calculations go through with $t_*$ replaced by $t_*'$,
etc., and corresponding new coordinates $x_*,x_{**}$.
Namely, using that \eqref{eq:r-t_*-t_**} holds with $t_*,t_{**}$
replaced by $t'_*,t'_{**}$, replacing $e^{\kappa
  t_*'}$ by $x_*'$,
\begin{equation*}\begin{aligned}
u_{**}&=(x_*') _+^\alpha e^{-\alpha\kappa t_{**}'}  e^{2\alpha\kappa\tilde\Phi(r)}
e^{2i\s\tilde\Phi+2ik\tilde\Psi}u_*\\
&=(x_*') _+^\alpha e^{i\s' t_{**}'}e^{-2i(\s'-\s)\tilde\Phi(r)+2ik\tilde\Psi}
u_*\\
&=(x_*') _+^\alpha e^{i\s' t_{**}'}e^{-2i\frac{ka}{r_+^2+a^2}\tilde\Phi(r)+2ik\tilde\Psi}
u_*,
\end{aligned}\end{equation*}
hence
\begin{equation*}\begin{aligned}
u&=e^{-i\s' t_{**}'-ik\phi_{**}'}u_{**}=e^{-i\s' t_{**}'-ik\phi_{**}'}(x_*') _+^\alpha e^{i\s' t_{**}'}e^{-2i\frac{ka}{r_+^2+a^2}\tilde\Phi(r)+2ik\tilde\Psi}
u_*\\
&=(x_*') _+^\alpha e^{-ik\phi_{**}'
  -i\frac{2ak}{r_+^2+a^2}\tilde\Phi(r)+2ik\tilde\Psi}u_*\\
&=(x_*') _+^\alpha e^{-ik\phi_{*}' }u_*,
\end{aligned}\end{equation*}
and the extension across the bifurcate sphere as a conormal
distribution is clear. Moreover, the non-separated version of the arguments presented here
in Section~\ref{sec:alternative-construction}, using homogeneity
considerations in $x_*$, with smooth dependence on the spherical
variables and $x_{**}$, show that the extension solves the wave equation in a full
neighborhood of the bifurcate sphere.

\section{The boundary pairing in frequency space}\label{sec:frequency-pairing}

\noindent
Since $\s \neq 0$, we see that the ordinary differential operator $\xi^{-s}\hat {\mathcal P} \xi^s$ has two regular singular points at
\[
	\xi = 0, -2\s.
\]
Without loss of generality, let us assume that $\s < 0$; for $\s>0$ we
work with $(-\infty,-2\s)$ in place of $(-2\s,\infty)$ below.
The point is now that by Proposition \ref{prop: u Fourier transform}, the function $\hat {\mathfrak u}(\xi) := \F(\mathfrak u)(\xi)$ is smooth on the interval
\[
	(-2\s, \infty),
\]
since it does not contain $\xi = 0$.

\begin{prop}[Boundary pairing in frequency space] \label{prop: boundary pairing}
Assume $\mathfrak u$ is a dual solution in $\Ker\A^\dagger=\Ker {\mathcal P} $.
Assume $\s < 0$.
Then
\[
	\Im \ldr{ \xi^{-2s} \hat {\mathcal P} \hat {\mathfrak u}, \hat {\mathfrak u}}_{L^2(-2\s, \infty)} 
		= e^{\Im\a\pi}\sqrt{m^2 - a^2} \abs{b}^2 + e^{\Im\a\pi} m (-2\s)^{2-2s} \abs{\hat {\mathfrak u}(-2\s)}^2,
\]
where
\[
	\a
		:= -\frac{i \left( \left( r_+^2 + a^2 \right) \s + ak \right) }{\sqrt{m^2 - a^2}} - s.
\]
\end{prop}
\begin{proof}
  First,
\begin{equation*}\begin{aligned}
 & 2i\Im\ldr{ \xi^{-2s} \hat {\mathcal P} \hat {\mathfrak u}, \hat
    {\mathfrak u}}_{L^2(-2\s, \infty)}=\ldr{ \xi^{-2s} \hat {\mathcal P} \hat {\mathfrak u}, \hat
    {\mathfrak u}}_{L^2(-2\s, \infty)}
  -\ldr{\hat {\mathfrak u},\xi^{-2s} \hat {\mathcal P} \hat {\mathfrak u}}_{L^2(-2\s, \infty)} \\
&=\ldr{ (\xi^{-s}\hat {\mathcal P} \xi^s)(\xi^{-s}\hat{\mathfrak
  u}), \xi^{-s}\hat{\mathfrak
  u}}_{L^2(-2\s, \infty)}-\ldr{ \xi^{-s}\hat{\mathfrak
  u} ,(\xi^{-s}\hat {\mathcal  P}\xi^s)(\xi^{-s}\hat{\mathfrak
  u}) }_{L^2(-2\s, \infty)}
  \end{aligned}\end{equation*}
  a priori vanishes when $\hat {\mathcal P} \hat{\mathfrak
  u}=0$: both terms on the right hand side are simply $0$. The key point is to compute this
difference a different way, namely using that $\xi^{-s}\hat {\mathcal
  P} \xi^s$ is symmetric on an appropriate domain, and using that when
symmetry fails (due to distributions not being in the domain or an
additional boundary, here $\xi=-2\s$, being introduced) one can
obtain a sum of non-negative terms. Even with the failure of symmetry
due to domain reasons,
many computations become easier since the symmetry implies that many
terms can be dropped.
  
To proceed, recall from Proposition~\ref{prop: second conjugation} that
  \begin{align*}
	\xi^{-s}\hat {\mathcal P} \xi^s
		=& \ - \d_\xi \left( \xi^2 + 2 \s \xi \right) \d_\xi - 2 mi \xi\d_\xi \xi + 2 \xi \left( a^2 \s + a k \right) \\
		& \ + \xi^2 a^2 + s^2\frac{\xi + 2\s} \xi + \lambda,
  \end{align*}
  which is symmetric on $\dot H^{1,1}([-2\sigma,\infty))$, where the
  dot refers to vanishing at $\xi=-2\s$.
  Moreover, as $\Re(\a)=-s$, by\footnote{The $L^2$ based spaces can be
    replaced by $L^\infty$ based ones using Sobolev embedding
    relative to the stronger $\xi\d_\xi$ derivatives, i.e.\ relative
    to $\eta=\log\xi$; this increases the decay weights by
    $1/2$. There is actually no need for this with our second approach
  as Remark~\ref{remark:stronger-dual-structure} gives a more precise structure.} Proposition~\ref{prop: u Fourier transform}, 
  $$
  \xi^{-s}\hat{\mathfrak
  u}\in\bigcap_{\e>0}e^{-ir_+\xi}\bar
  H_*^{\infty,1/2-\e}([-2\sigma,\infty)),
  $$
  and moreover differs from $\xi^{-s}b e^{-ir_+\xi} e^{-i(\a+1)\pi/2}\left( \xi - i
  \right)^{-\a-1} $ by an element of $\bigcap_{\e>0}e^{-ir_+\xi}\bar
  H_*^{\infty,3/2-\e}([-2\sigma,\infty))$.
    Correspondingly, the only reasons for the potential non-vanishing of
  $$
\ldr{ (\xi^{-s}\hat {\mathcal P} \xi^s)(\xi^{-s}\hat{\mathfrak
  u}), \xi^{-s}\hat{\mathfrak
  u}}_{L^2(-2\s, \infty)}-\ldr{ \xi^{-s}\hat{\mathfrak
  u} ,(\xi^{-s}\hat {\mathcal P} \xi^s)(\xi^{-s}\hat{\mathfrak
  u}) }_{L^2(-2\s, \infty)}
$$
are the non-vanishing of $\hat{\mathfrak
  u}$ at $\xi=-2\sigma$, and the slower than necessary decay of $\hat{\mathfrak
  u}$ at
infinity. However, introducing a cutoff $\phi\in C^\infty_c(\R)$,
identically $1$ on $[-1,1]$, and letting $\phi_R(\xi)=\phi(\xi/R)$ we have
\begin{equation}\begin{aligned}\label{eq:symmetry-failure}
&\ldr{ \phi_R(\xi)(\xi^{-s}\hat {\mathcal P} \xi^s)(\xi^{-s}\hat{\mathfrak
  u}), \xi^{-s}\hat{\mathfrak
  u}}_{L^2(-2\s, \infty)}-\ldr{ \phi_R(\xi)\xi^{-s}\hat{\mathfrak
  u} ,(\xi^{-s}\hat {\mathcal P} \xi^s)(\xi^{-s}\hat{\mathfrak
  u}) }_{L^2(-2\s, \infty)}\\
&=\ldr{ [\phi_R(\xi),(\xi^{-s}\hat {\mathcal P} \xi^s)](\xi^{-s}\hat{\mathfrak
  u}), \xi^{-s}\hat{\mathfrak
  u}}_{L^2(-2\s, \infty)}\\
&\qquad+\Big(\ldr{ (\xi^{-s}\hat {\mathcal P} \xi^s)\phi_R(\xi)(\xi^{-s}\hat{\mathfrak
  u}), \xi^{-s}\hat{\mathfrak
  u}}_{L^2(-2\s, \infty)}-\ldr{ \phi_R(\xi)\xi^{-s}\hat{\mathfrak
  u} ,(\xi^{-s}\hat {\mathcal P} \xi^s)(\xi^{-s}\hat{\mathfrak
  u}) }_{L^2(-2\s, \infty)}\Big).
\end{aligned}\end{equation}
Now, $\phi_R$ is uniformly bounded in symbols of order $0$, and tends
to $1$ as $R\to\infty$ in $S^\e$ for $\e>0$, thus the expression on the first line of the right hand side tends to
$0$ as $R\to\infty$ if $\xi^{-s}\hat{\mathfrak
  u}$ in one of the slots is replaced by an element of $\bar
H^{1,1}([-2\sigma,\infty))$, and in the other by an element of $\bar
H^{1,0}([-2\sigma,\infty))$ (there is also a gain in
differentiability, but this is not relevant) because the commutator is
uniformly bounded as a map from one of these spaces to the dual of the
other (thanks to the gain in decay in the commutator, uniformly in $R$) and tends to $0$
on a dense subset, thus strongly, so in the limit as $R\to\infty$ this
term in fact equals the limit of
\begin{equation}\label{eq:symmetry-failure-at-infinity}
\ldr{ [\phi_R(\xi),(\xi^{-s}\hat {\mathcal P} \xi^s)]b e^{-ir_+\xi} e^{-i(\a+1)\pi/2}\xi^{-s}\left( \xi - i
  \right)^{-\a-1}, b e^{-ir_+\xi} e^{-i(\a+1)\pi/2}\xi^{-s}\left( \xi - i
  \right)^{-\a-1}}_{L^2(-2\s, \infty)},
\end{equation}
and indeed the $-i$ can be dropped in $\left( \xi - i
  \right)^{-\a-1}$ for the same reason (keeping in mind that if we had support in $\xi<0$, this means
  $e^{-i\pi  (-\a-1)}|\xi|=e^{i\pi  (\a+1)}|\xi|$ there, while it is
  just $|\xi|^{-\a-1}$ in $\xi>0$).
Moreover, by similar considerations, all terms of $\xi^{-s}\hat
{\mathcal P} \xi^s$ with subleading growth in $\xi$ can be dropped,
i.e.\ the operator can be replaced by
$$
-\d_\xi \xi^2\d_\xi-2mi\xi\d_\xi\xi+\xi^2a^2
$$
in the computation, and the last term can be dropped as it commutes
with $\phi_R$. Now, $[\xi\d_\xi,\phi_R]=(\xi/R)\phi'(\xi/R)$,
$\d_\xi\xi=\xi\d_\xi+1$, so
$$
[\phi_R(\xi), -\d_\xi
\xi^2\d_\xi-2mi\xi\d_\xi\xi]=(\xi/R)\phi'(\xi/R)\xi\d_\xi+\d_\xi\xi
(\xi/R)\phi'(\xi/R)+2mi \xi (\xi/R)\phi'(\xi/R).
$$
Substituting this into \eqref{eq:symmetry-failure-at-infinity} with $\left( \xi - i
  \right)^{-\a-1}$ replaced by $\xi^{-\a-1}$ for the reasons mentioned
  above,
shifting $\d_\xi\xi$ in its second term on the right hand side to the
second slot of \eqref{eq:symmetry-failure-at-infinity} as $-\xi
\d_\xi$, we observe that if $\xi \d_\xi$ hit $\xi^{-s-\a-1}$ rather than $e^{-ir_+\xi}$ we have an additional
  factor of $\xi$-decay (and $\xi^{-1}=(\xi/R)^{-1}R^{-1}$) yielding
  $0$ in the limit, we obtain
  \begin{equation*}\begin{aligned}
&|b|^2 (-2ir_++2mi)e^{\Im\alpha\pi}\lim_{R\to\infty}\int_{-2\s}^\infty
(\xi/R)\phi'(\xi/R)\xi|\xi^{-s-\a-1}|^2\,d\xi\\
&=|b|^2 (-2ir_++2mi) e^{\Im\alpha\pi}\lim_{R\to\infty}\int_{-2\s}^\infty
(\xi/R)\phi'(\xi/R)\xi^{-1}\,d\xi\\
&=|b|^2 (-2ir_++2mi) e^{\Im\alpha\pi}\lim_{R\to\infty}\int_{-2\s}^\infty
(\xi/R)\phi'(\xi/R)\xi^{-1}\,d\xi\\
&=|b|^2 (-2ir_++2mi) e^{\Im\alpha\pi}\int_0^\infty\phi'=2i|b|^2(r_+-m) e^{\Im\alpha\pi}.
  \end{aligned}\end{equation*}

On the other hand, for the expression on the second line as the
support in one of the slots is compact, the only reason for
non-vanishing is the boundary term at $\xi=-2\s$, which can be simply
computed. In fact, only the derivative terms contribute, and
with $v=\phi_R(\xi)\xi^{-s}\hat{\mathfrak
  u} $, $w=\xi^{-s}\hat{\mathfrak
  u}$, $R$ sufficiently large so that $\phi_R\equiv 1$ near $-2\s$,
\begin{equation*}\begin{aligned}
&\ldr{ (- \d_\xi \left( \xi^2 + 2 \s \xi \right) \d_\xi - 2 mi
  \xi\d_\xi \xi)v,w}_{L^2(-2\s, \infty)}
=\int_{-2\s}^\infty  (- \d_\xi \left( \xi^2 + 2 \s \xi \right) \d_\xi - 2 mi
\xi\d_\xi \xi)v\overline{w}\,d\xi\\
&=-\left(\left( \xi^2 + 2 \s \xi \right) \d_\xi v\right)\overline{w}|_{-2\s}-2mi\xi
v\xi\overline{w}|_{-2\s}
+\int_{-2\s}^\infty\left(\left( \xi^2 + 2 \s \xi \right) \d_\xi
  v\overline{\d_\xi w}-v\overline{2mi \xi \d_\xi \xi w}\right)\,d\xi\\
&=2mi(-2\s)^2v(-2\s)\overline{w(-2\s)}+v \left( \xi^2 + 2 \s \xi \right) \overline{\d_\xi w}|_{-2\s}
+\int_{-2\s}^\infty\left(
  -v\overline{\d_\xi\left( \xi^2 + 2 \s \xi \right)  \d_\xi
    w}-v\overline{2mi \xi \d_\xi \xi w}\right)\,d\xi\\
&=2mi(-2\s)^2v(-2\s)\overline{w(-2\s)}+\ldr{ v ,(- \d_\xi \left( \xi^2 + 2 \s \xi \right) \d_\xi - 2 mi
  \xi\d_\xi \xi)w) }_{L^2(-2\s, \infty)},
\end{aligned}\end{equation*}
so substituting in $v,w$, the second line in
\eqref{eq:symmetry-failure} becomes
$$
2mi(-2\s)^{2-2s}|\hat{\mathfrak
  u} (-2\s)|^2.
$$
Combining this with the computation of the first line of
\eqref{eq:symmetry-failure} and recalling that 
\[
	r_+ = m + \sqrt{m^2 - a^2},
      \]
      the proposition follows.

We also give a second closely related argument using $L^\infty$ versions of
the structure of the Fourier transformed dual states in the
pairings taking advantage of Remark~\ref{remark:stronger-dual-structure}.
Since all coefficients are real, the zero order part of this operator will not contribute to the boundary pairing. 
We thus only need to consider the first two terms - those that involve derivatives.
The first term gives the contribution
\begin{align*}
	\Im 
		& \ldr{- \d_\xi \left( \xi^2 + 2 \s \xi \right) \d_\xi \xi^{-s} \hat {\mathfrak u}, \xi^{-s} \hat {\mathfrak u}}_{L^2(-2\s, \infty)} \\
		= & \ - \Im \int_{-2\s}^\infty \d_\xi \left( \left( \xi^2 + 2\s \xi \right) \left( \d_\xi \xi^{-s} \hat {\mathfrak u} \right) \overline{\xi^{-s} \hat {\mathfrak u}} \right) \md \xi \\
		& \ + \Im \ldr{ \left( \xi^2 + 2\s \xi \right) \d_\xi \xi^{-s} \hat {\mathfrak u}, \d_\xi \xi^{-s} \hat {\mathfrak u}}_{L^2(-2\s, \infty)} \\
		= & \ - \Im \left( \lim_{\xi \to \infty} \left( \xi^2 + 2\s \xi \right) \left( \d_\xi \xi^{-s} \hat {\mathfrak u} \right) \overline{ \xi^{-s} \hat {\mathfrak u}} - \left( \xi^2 + 2\s \xi \right) \left( \d_\xi \xi^{-s} \hat {\mathfrak u} \right) \overline{\xi^{-s}\hat {\mathfrak u}} |_{\xi = -2\s} \right) \\
		= & \ - \Im \lim_{\xi \to \infty} \left( \xi^2 + 2\s \xi \right) \left( \d_\xi \xi^{-s} \hat {\mathfrak u} \right) \overline{\xi^{-s} \hat {\mathfrak u}} \\
		= & \  r_+ \lim_{\xi \to \infty} \xi^2 \abs{b e^{-ir_+ \xi} e^{-i(\a+1)\pi/2} \left(  \xi - i \right)^{-\a-1} \xi^{-s}}^2 \\
		= & \ r_+ e^{\Im\a \pi}\abs{b}^2.
\end{align*}
The second term gives the contribution
\begin{align*}
	-\Im \ldr{2mi \xi \d_\xi \xi \xi^{-s} \hat {\mathfrak u}, \xi^{-s} \hat {\mathfrak u}}_{L^2(-2\s, \infty)}
		= & -\ 2m \Re \ldr{\d_\xi \xi^{1-s} \hat {\mathfrak u}, \xi^{1-s} \hat {\mathfrak u}}_{L^2(-2\s, \infty)} \\
		= & -\ m \int_{-2\s}^\infty \d_\xi \abs{\xi^{1-s} \hat {\mathfrak u}}^2 \md \xi \\
		= & -\ m \left[ \abs{ \xi^{2-2s} \hat {\mathfrak u}(\xi)}^2 \right]_{-2\s}^\infty \\
		= & -\ m \lim_{\xi \to \infty} \xi^{2-2s} \abs{\hat {\mathfrak u}(\xi)}^2 +  m (-2\s)^{2-2s} \abs{\hat {\mathfrak u}(-2\s)}^2 \\
		= & -\ m  e^{\Im\a \pi}\abs{b}^2 + m   e^{\Im\a \pi} (-2\s)^{2-2s} \abs{\hat {\mathfrak u}(-2\s)}^2.
\end{align*}
This completes the proof as above.
\end{proof}

\begin{remark}\label{remark:freq-pm-infty}
  It is worthwhile computing what the contribution of a term
  $$
  \ldr{ [\phi_R(\xi),(\xi^{-s}\hat {\mathcal P} \xi^s)](\xi^{-s}\hat{\mathfrak
  u}), \xi^{-s}\hat{\mathfrak
  u}}_{L^2(-\infty, \infty)}
$$
would look like near $-\infty$ (the $+\infty$ computation is
above). The same argument yields, keeping in mind that in $\xi<0$,
$(\xi-i)^{-\a-1}$ should be replaced by $e^{\pi(\a+1)}|\xi|$ in the
computation,
 \begin{equation*}\begin{aligned}
&|b|^2 (-2ir_++2mi)e^{\Im\alpha\pi}|e^{i\pi  (\a+1)}|^2\lim_{R\to\infty}\int_{-\infty}^0
(\xi/R)\phi'(\xi/R)\xi|\xi^{-s-\a-1}|^2\,d\xi\\
&=|b|^2 (-2ir_++2mi) e^{-\Im\alpha\pi}\lim_{R\to\infty}\int_{-\infty}^0
(\xi/R)\phi'(\xi/R)\xi^{-1}\,d\xi\\
&=|b|^2 (-2ir_++2mi) e^{-\Im\alpha\pi}\lim_{R\to\infty}\int_{-\infty}^0
(\xi/R)\phi'(\xi/R)\xi^{-1}\,d\xi\\
&=|b|^2 (-2ir_++2mi) e^{-\Im\alpha\pi}\int_{-\infty}^0\phi'=-2i|b|^2 e^{-\Im\alpha\pi} (r_+-m),
\end{aligned}\end{equation*}
so the term has the opposite sign from the $\xi>0$ contribution, as
expected, with otherwise almost the same coefficient; the combination
of the two terms yields
$$
(e^{\Im\a\pi}-e^{-\Im\a \pi})\sqrt{m^2 - a^2} \abs{b}^2,
$$
which is non-negative if $\Im\a\geq 0$, i.e.\ if
$(r_+^2+a^2)\s+ak<0$. While we did not compute the contribution from
the singularity of $\hat{\mathfrak u}$ at $\xi=0$, this arises from the
source at $r=\infty$, thus it would also give a non-negative
contribution. Hence the global pairing in frequency space (which then
could also have been done in position space) gives a useful result,
with terms of matching signs, exactly in the case of no superradiance.
  \end{remark}

\noindent
We may finally prove the main result, Theorem \ref{thm: QNM vanishes}:

\begin{proof}[Proof of Theorem \ref{thm: QNM vanishes}] \ \\
By construction of $\hat {\mathcal P}$ and $\hat {\uf}$, 
\[
	\hat {\mathcal P} \hat {\uf}
		= \F \left( {\mathcal P} \uf \right)
		= 0.
\]
Proposition~\ref{prop: boundary pairing} in particular implies that
$b= 0$. By the microlocal regularity, Lemma~\ref{lemma:dual-solution-principal-structure}, $\uf$
is thus $C^\infty$, and by its support property it vanishes with all
derivatives at $r_+$.
Since $r_+$ is a regular singular point, a standard energy estimate
then implies that $\uf= 0$ for all $r \in (r_+-\de, \infty)$ for a suitably small $\de > 0$.
It follows that $\uf = 0$ as claimed.
\end{proof}

{\em In case we do not separate variables, the same arguments go
  through by simply also adding integration on the sphere, and a PDE
  version of the unique contuinuation in the last step.}

\section{A comparison with the Klein-Gordon equation}

\noindent
Note that
\[
	\mathrm L_0
		= (r^2 + a^2 \cos^2(\theta))\Box,
\]
so Theorem~\ref{thm: QNM vanishes} also covers quasinormal mode operators for the scalar wave equation.
In this section, we would like to illustrate where the above proof of
Theorem~\ref{thm: QNM vanishes} goes wrong for the Klein-Gordon
equation with a sufficiently large mass relative to $|\s|$ to
emphasize the difference between the wave/Teukolsky and Klein-Gordon
equations.

Let us therefore consider the modified operator
\[
	(r^2 + a^2 \cos^2(\theta))\left( \Box + M^2\right),
\]
for some Klein-Gordon mass $M \neq 0 \in \R$.
The point is that the expression corresponding to Proposition~\ref{prop: second conjugation} would now become
\[
	\hat {\mathcal P}
		= - \d_\xi \left( \xi^2 + 2 \s \xi + M^2 \right) \d_\xi - 2 mi \xi\d_\xi \xi + 2 \xi \left( a^2 \s + a k \right) + \xi^2 a^2 + \lambda.
\]
This new operator does not have any regular singular points if
\[
	\abs{M} 
		\geq \abs{\s},
\]
hence the boundary pairing in frequency space on $[-2\s,\infty)$ would
not go through.

On the other hand, the boundary pairing on $\R$ in position (or equivalently, frequency space)
{\em does} go through, but now the ellipticity of the operator
$\mathcal P$ at $r=\infty$, thus of $\hat {\mathcal P}$ for finite $\xi$, means
that $\hat{\mathfrak
  u}$ has no local singularities (thus there is no analogue of the
$\xi=0$ singularity in the $M=0$ case), while the contributions from
$r=r_+$ work out just as in Remark~\ref{remark:freq-pm-infty}, so when
$\Im\a\neq 0$, i.e.\ when $(r_+^2+a^2)\s+ak\neq 0$ one concludes that
any mode solution necessarily vanishes. This is of course closely
related to the boundary pairing on $(r_+,\infty)$ in
position space, which can be found in the work of Shlapentokh-Rothman
\cite{S2014}. In this case one obtains a generally non-trivial
contribution at $r_+$, with prefactor $(r_+^2+a^2)\s+ak$ (in our
notation), hence $(r_+^2+a^2)\s+ak$ must vanish if a real Klein-Gordon
mode exists. What \cite{S2014} proves is that such modes do exist, and
can become growing modes upon varying the parameters.

\section{A comparison with Whiting's transform}\label{sec:Whiting}
The Whiting transform of a (direct) mode solution, in our notation,
following \cites{AMPW2017,S2015} and restricting to $s=0$ for
simplicity (the general case being similar), takes
the form
\begin{equation*}\begin{aligned}
\tilde
u(\tilde\xi)=&(\tilde\xi^2+a^2)^{1/2}(\tilde\xi-r_+)^{-2mi\s}e^{-i\s\tilde\xi}\\
&\qquad\int_{r_+}^\infty
e^{\frac{2i\s}{r_+-r_-}(\tilde\xi-r_-)(r-r_-)}(r-r_-)^\eta(r-r_+)^\zeta
  e^{-i\s r}u(r) \,dr
    \end{aligned}\end{equation*}
  with
  \begin{equation*}\begin{aligned}
      \eta&=\frac{i(ak+(r_-^2+a^2)\s)}{r_+-r_-},\\
      \zeta&=\frac{-i(ak+(r_+^2+a^2)\s)}{r_+-r_-}.
    \end{aligned}\end{equation*}
  This can be rewritten as
 \begin{equation*}\begin{aligned}
\tilde
u(\tilde\xi)=&(\tilde\xi^2+a^2)^{1/2}(\tilde\xi-r_+)^{-2mi\s}e^{-i\s\tilde\xi}\\
&\qquad e^{\frac{-2i\s r_-}{r_+-r_-}\tilde\xi}e^{\frac{2i\s r_-^2}{r_+-r_-}}\int_{r_+}^\infty
e^{\frac{2i\s}{r_+-r_-}\tilde\xi r}e^{\frac{-2i\s r_-}{r_+-r_-}r}(r-r_-)^\eta(r-r_+)^\zeta
  e^{-i\s r}u(r) \,dr
\end{aligned}\end{equation*}
Up to a change of variables of the output to
$$
\hat \xi=\frac{-2\s}{r_+-r_-}\tilde\xi,
$$
and up to a multiplication by an appropriate factor,
this is a Fourier transform of a multiple of $u$, cut off at $r_+$,
i.e.\ multiplied by the characteristic function of
$[r_+,\infty)$. Moreover, by basic properties of the Fourier
transform (or directly combining the two exponentials), the factor $e^{\frac{-2i\s r_-}{r_+-r_-}r}$
  in the integral simply gives a translation in $\tilde\xi$, thus in
  $\hat\xi$. Hence the key question in comparing our approach to
  Whiting's is a computation of the singular factor
  $$
e^{G(r)}=(r-r_-)^\eta(r-r_+)^\zeta
  e^{-i\s r};
  $$
  thus
  $$
G'(r)=\frac{\eta}{r-r_-}+\frac{\zeta}{r-r_+}-i\s.
$$
Some algebraic manipulation after putting all terms on common
denominator $\mu=(r-r_+)(r-r_-)$ yields $G'=-H'$, so in fact this is
indeed the Fourier transform of the adjoint solution as that restricts
to $e^{-H(r)}u$ in $(r_+,\infty)$ and is supported in
$[r_+,\infty)$. An advantage of our distribution theoretic framework,
as well as the direct use of the Fourier transform with well known
properties, is that all of the formal computations are justified and fall
into a conceptual framework.

\appendix

\section{The boundary pairing in position space}\label{appendix:position-pairing}

\noindent
For comparison we include a description of the standard position space
pairing; see e.g.\ \cite{S2015}.
Let $f: (r_+ - \de, \infty) \to \R$ be a smooth function such that $f(r) = 1$ for $r \geq r_+ + 2$ and $f(r) = - 1$ for $r \leq r_+ + 1$.
Defining
\[
	w(r)
		:= u(r) e^{h(r)},
\]
where $h'(r) = f(r) H'(r)$.
It follows that $w(r)$ is smooth in $r$ on $[r_+, \infty)$ and smooth in $r^{-1}$ on $[r_+, \infty]$.
Moreover, with $\tilde \P := e^h \P e^{-h}$, we have
\[
	\tilde \P w(r)
		= e^h \P u(r)
		= 0
\]
for all $r \in [r_+, \infty)$.

\noindent
Lemma~\ref{le: conjugation} implies with $x = r - r_+$ that
\begin{align*}
	\tilde \P
		= & \ - \d_r \mu(r) \d_r - \d_r f(r) \left( i \left( \left( r^2 + a^2 \right) \s + ak \right)  + (r - m)s \right) \\
		& \ - \left( i \left( \left( r^2 + a^2 \right) \s + ak \right) + (r - m)s \right) f(r) \d_r - 4 s i r \s + \lambda \\
		& \ + (f^2 - 1) \frac{\left( i \left( \left( r^2 + a^2 \right) \s + ak \right) + (r - m)s \right)^2}{\mu(r)},
\end{align*}
which smoothly extends to $r_+$ by the conditions on $f$.
The following is the boundary pairing in physical space:

\begin{prop}[Boundary pairing in physical space]
Assume that $s = 0$ and that $w: [r_+, \infty)$ is smooth in $r^{-1}$ on $[r_+, \infty]$.
Assume moreover that $\lim_{r \to \infty} w(r) = 0$.
Then
\[
	\Im \ldr{\tilde \P w, w}_{L^2(r_+, \infty)}
		= \lim_{b \to \infty}b^2 \abs w^2(b) + \left( \left( r_+^2 + a^2 \right) \s + ak \right) \abs w^2(r_+).
\]
\end{prop}
\begin{proof}
We compute
\begin{align*}
	\ldr{-\d_r \mu \d_r w, w}_{L^2(r_+, b)}
		= & \ \int_{r_+}^b \left( - \d_r \mu \d_r w \right) \overline w \md r \\
		= & \ \left[ \left( \mu(r) \d_r w \right) \overline w \right]^b_{r_+} + \int_{r_+}^b \mu \abs{\d_r w}^2 \md r \\
		= & \ \left( \mu(r) \d_r w \right)(b) \overline{w(b)} + \int_{r_+}^b \mu \abs{\d_r w}^2 \md r.
\end{align*}
By assumption, $w(r)$ is smooth in $r^{-1}$ on $[r_+, \infty]$.
Defining $\rho := \frac1r$, we note that $r^2 \d_r = - \d_{\rho}$.
Smoothness in $\rho$ therefore implies that $r^2 \d_r w$ is bounded.
Since by assumption $\lim_{r \to \infty} w(r) = 0$, we conclude that
\[
	\Im \left( \ldr{-\d_r \mu \d_r w, w}_{L^2(r_+, b)} \right) 
		= 0.
\]
Next, we compute
\begin{align*}
	& \ldr{ - \d_r f(r) \left( i \left( \left( r^2 + a^2 \right) \s + ak \right) \right)w, w}_{L^2(r_+, b)} \\
	& - \ldr{\left( i \left( \left( r^2 + a^2 \right) \s + ak \right) \right) f(r) \d_r w, w}_{L^2(r_+, b)} \\
		= & \ 2 \int_{r_+}^b \d_r \left( i f(r) \left( \left( r^2 + a^2 \right) \s + ak \right) \abs w^2 \right) \md r \\
		& \ + \ldr{ w, - \d_r f(r) \left( i \left( \left( r^2 + a^2 \right) \s + ak \right) \right) w }_{L^2(r_+, b)} \\
		& \ - \ldr{w, \left( i \left( \left( r^2 + a^2 \right) \s + ak \right) \right) f(r) \d_r w}_{L^2(r_+, b)},
\end{align*}
where
\begin{align*}
	& \int_{r_+}^b \d_r \left( i f(r) \left( \left( r^2 + a^2 \right) \s + ak \right) \right) \abs w^2 \md r \\
		& \ = i \left( \left( b^2 + a^2 \right) \s + ak \right) \abs w^2(b) + i \left( \left( r_+^2 + a^2 \right) \s + ak \right) \abs w^2(r_+),
\end{align*}
if $b \geq 2$, implying that $f(b) = - 1$.
Inserting these computations proves the statement, as the remaining terms are real and of order $0$.
\end{proof}

\end{sloppypar}
\end{document}